\documentclass [14pt]{amsart}
\usepackage{amsmath}
\usepackage{amssymb}
\usepackage{amsfonts}
\usepackage{amscd}
\usepackage{epsfig}
\usepackage{amsxtra}
\usepackage{pifont}
\usepackage{mathabx}
\usepackage{MnSymbol}
\usepackage{accents}
\usepackage{scalerel,stackengine}
\allowdisplaybreaks
\stackMath
\newcommand\reallywidehat[1]{%
\savestack{\tmpbox}{\stretchto{%
  \scaleto{%
    \scalerel*[\widthof{\ensuremath{#1}}]{\kern-.6pt\bigwedge\kern-.6pt}%
    {\rule[-\textheight/2]{1ex}{\textheight}}
  }{\textheight}%
}{0.5ex}}%
\stackon[1pt]{#1}{\tmpbox}%
}

\newcommand\reallywidecheck[1]{%
\savestack{\tmpbox}{\stretchto{%
  \scaleto{%
    \scalerel*[\widthof{\ensuremath{#1}}]{\kern-.6pt\bigwedge\kern-.6pt}%
    {\rule[-\textheight/2]{1ex}{\textheight}}
  }{\textheight}%
}{0.5ex}}%
\stackon[1pt]{#1}{\scalebox{-1}{\tmpbox}}%
}

\numberwithin{equation}{section}

\allowdisplaybreaks

\newcommand{\supp}{\mbox{\rm supp}}
\newcommand{\dens}{\mbox{\rm dens}}
\newcommand{\udens}{\overline{\mbox{\rm dens}}}
\newcommand{\ldens}{\underline{\mbox{\rm dens}}}

\newcommand{\RR}{{\mathbb R}}

\newcommand{\CC}{{\mathbb C}}
\newcommand{\TT}{\mathbb T}
\newcommand{\NN}{\mathbb N}

\newcommand{\XX}{\mathbb X}
\newcommand{\YY}{\mathbb Y}
\newcommand{\AAA}{\mathbb A}
\newcommand{\SSS}{\mathbb S}

\newcommand{\cL}{{\mathcal L}}

\newcommand{\uM}{\overline{M}}

\newcommand{\oplam}{\mbox{\Large $\curlywedge$}}

\newcommand{\cM}{{\mathcal M}}

\newcommand{\cA}{{\mathcal A}}
\newcommand{\cB}{{\mathcal B}}
\newcommand{\cT}{{\mathcal T}}

\newcommand{\card}{\mbox{\rm card}}
\newcommand{\eps}{\varepsilon}
\newcommand{\m}{\mathfrak{ m}}

\newcommand{\dd}{\mbox{\rm d}}

\newcommand{\Cu}{C_{\mathsf{u}}}
\newcommand{\Cc}{C_{\mathsf{c}}}

\newcommand{\SAP}{\mathcal{SAP}}
\newcommand{\Bap}{\mathcal{B}\hspace*{-1pt}{\mathsf{ap}}}

\newcommand{\extW}{\cM_W^G}
\newcommand{\extB}{\cM_B^G}
\newcommand{\extU}{\cM_U^G}

\theoremstyle{definition}
 \newtheorem{theorem}{Theorem}[section]
 \newtheorem{lemma}[theorem]{Lemma}
 \newtheorem{proposition}[theorem]{Proposition}
 \newtheorem{corollary}[theorem]{Corollary}

 \newtheorem{question}[theorem]{Question}
 \newtheorem{definition}[theorem]{Definition}
 
  \newtheorem{remark}[theorem]{Remark}

\begin{document}

\title[Borel model sets]{Model sets with precompact Borel windows}

\author{Nicolae Strungaru}
\address{Department of Mathematical Sciences, MacEwan University \\
10700 -- 104 Avenue, Edmonton, AB, T5J 4S2\\
Phone: 780-633-3440 \\
and \\
Institute of Mathematics ``Simon Stoilow''\\
Bucharest, Romania}
\email{strungarun@macewan.ca}
\urladdr{http://academic.macewan.ca/strungarun/}

\begin{abstract} Given a cut and project scheme and a pre-compact Borel window we show that almost surely all positions of the window
give rise to point sets with Besicovitch almost periodic Dirac combs. In particular, all those positions lead to pure point diffractive point sets
with autocorrelation given by the covariogram of the window and diffraction given by the square of the Fourier transform of the window. Moreover the Fourier--Bohr
coefficients exist and the intensity of the Bragg peaks is given by their absolute value square. We show the existence of an ergodic measure with pure point dynamical
spectrum such that all those window positions give rise to point sets which are generic for this measure. We complete the paper by reanalysing the class of
weak model sets of maximal density.
\end{abstract}



\maketitle

\section{Introduction}

In 1984 Dan Schettmann discovered new solids with clear diffraction pattern and 5-fold symmetry \cite{She}. This discovery led to a paradigm shift in physics, as for more that 100 years clear diffraction patterns have been thought to be the identifying signature of periodic crystals, and thus incompatible with 5-fold rotational symmetry. The newly discovered materials are now called quasi-crystals.

When trying to label the positions of Bragg peaks in the diffraction of the quasicrystals physicists observed that these positions can be typically indexed as linear combinations with integer coefficients of 6 fixed vectors (instead of 3 vectors as for periodic crystals). This suggests that maybe quasicrystals can be modeled by using lattices in higher dimensional spaces.

The standard method for producing mathematical models for quasicrystals is via a cut and project scheme(or on short CPS). This method was introduced in 1970's by Y. Meyer \cite{Meyer} in order to produce point sets with certain harmonic properties, and was rediscovered by deBruijn \cite{dBr} when studying the construction of the Penrose tiling \cite{Pen} via the Ammann bars. Meyer's pioneering work in the study of CPSs was popularized and expanded in the 1990's by R.V. Moody \cite{MOO,Moody} and J. Lagarias \cite{LAG1,LAG}. The basic idea beyond CPSs is the following: to produce a model for a 3-dimensional quasicrystals we start with a lattice $\cL$ in a higher dimensional space which sits at an irrational angle with respect to our space, cut a strip of the lattice which is within bounded distance of the 3-dimensional space and project the lattice points in this strip on our space. The width of the strip is usually called the ``window" (see Def.~\ref{def:cps} below for the formal definition).

If the shape of the strip is not too bad, the point set produced via a CPS should still show some of residues of the long-range order of the lattice $\cL$.
Indeed, for polyhedral windows, Hof \cite{Hof1,Hof2,HOF3} and Solomyak \cite{So98} showed that the resulting point sets show a clear diffraction pattern (pure point diffraction), and the intensities of the Bragg peaks can be related to the Fourier transform of the window via the so called dual CPS (see below for definitions and the exact formulas). Pure point diffraction has been established by M. Schlottmann for a large class of windows, called regular, via the use of dynamical systems \cite{Martin2}. M. Baake and R.V. Moody gave an alternate proof of this result via the use of harmonic analysis and almost periodic measures \cite{BM}. Recently, C. Richard and I showed that the diffraction formula for regular model sets is just the Poisson Summation Formula (PSF) for the lattice $\cL$ in the super-space $G \times H$. In general, if the window is ``bad", the diffraction can have a continuous component but the pure point component still shows high coherence \cite{NS1,NS2,NS5,NS11,NS19}.

In recent years many results about model sets have been extended to a larger class of such models, called weak model sets of maximal density \cite{BHS,KR,KR2,Kel1}. This class contains many
models of number theoretic origin, such as square free integers, visible points of a lattice, $\cB$-free integers and coprime families of lattices (see for example \cite{BH,BMP,BKKL,KKL,PH}).
A key result in this settings is a theorem by R. V. Moody \cite{Moody2} which states that for any CPS and any compact window, almost surely all positions of the window give maximal density weak model sets, and hence point sets wit pure point diffraction. Similar results can be shown for open pre-compact windows (\cite{BHS}), but it was it clear so far how to extend these results beyond topological windows. It is our goal to do this in this paper, by using some new developments into the theory of Besicovitch almost periodic measures \cite{LSS,LSS2}.

Recently, \cite{LSS,LSS2} studied the connection between pure point diffraction and dynamical spectra and mean, Besicovitch and Weyl almost periodicity. It was shown that an ergodic dynamical system of translation bounded measures has pure point spectrum if and only if almost surely all elements are Besicovitch almost periodic (or mean almost periodic, respectively). Furthermore, each Besicovitch almost periodic measure $\mu$ gives rise to an ergodic measure $m$ with pure point dynamical spectrum, such that $\mu$ is generic for $m$. Finally, it was shown that maximal density weak model sets give rise to Besicovitch almost periodic Dirac combs.

In this paper, by combining the results and techniques from \cite{LSS,LSS2} with \cite{Moody2} and \cite{BHS,KR} we extend many results about maximal density weak model set to model sets with pre-compact Borel windows. Given a CPS $(G,H,\cL)$, a precompact Borel window $B$, and a tempered van Hove sequence $\cA$, we prove the existence of a set $\cT \subseteq \TT$ of full Haar measure such that, for all $(s,t)+\cL \in \cT$ the set $-s+\oplam(t+W)$ is Besicovitch almost periodic with respect to $\cA$ and has autocorrelation $\gamma$ and diffraction $\widehat{\gamma}$ (with respect to $\cA$) given by
\begin{align}
  \gamma &= \dens(\cL) \omega_{c(B)} = \dens(\cL) \sum_{(x,x^star) \in \cL} c(B)(x^\star) \delta_x \label{eqaut}\\
  \widehat{\gamma} &=\dens(\cL)^2 \omega^{}_{|\widecheck{1_B}|^2} = \dens(\cL)^2 \sum_{(\chi,\chi^\star) \in \cL^0} \left| \int_{B} \chi^\star(t) \dd t \right|^2 \delta_{\chi}  \,, \label{eqdiff}
\end{align}
where $c(B)$ denotes the covariogram of $B$ (see below for definitions and notations). In particular, for all $(s,t)+\cL \in \cT$ the set $-s+\oplam(t+B)$ has pure point diffraction spectra. We further show that for all $(s,t)+\cL \in \cT$ the set $-s+\oplam(t+B)$ has well defined Fourier--Bohr coefficients $a_{\chi}^{\cA}$ with respect to $\cA$, which satisfy
\[
a_{\chi}^{\cA}(-s+\oplam(t+B)) = \left\{
\begin{array}{cc}
  \chi(s) \int_{t+B} \chi^\star(z) \dd z & \mbox{ if } \chi \in \pi_{\widehat{G}}(\cL^0) \\
  0 & \mbox{ otherwise } \,.
\end{array}
\right.
\]
In particular, the Consistent Phase Property (CPP)
\[
\widehat{\gamma}(\{ \chi \})=\left| a_\chi^{\cA} \right|^2  \qquad \forall \chi \in \widehat{G}
\]
holds for all $(s,t)+\cL \in T$.

Next, we show that there exists an ergodic measure $\m$ on the extended hull $\extB = \overline{ \{ -s+\oplam(t+B) : s \in G, t \in H \}}$ such that, for all $(s,t)+\cL \in \cT$ the set $-s+\oplam(t+B)$ is generic for $\m$. In particular, $\m$ has pure point spectra, autocorrelation $\gamma$ and diffraction $\widehat{\gamma}$ given by \eqref{eqaut} and \eqref{eqdiff}, respectively. We want to emphasize here that in the case of compact and open windows, the measure $\m$ can be identified as the measure which on the hull leads to the extremal density of pointsets. This extremal condition implies that $\m$ is an extremal point of the set of $G$-invariant probability measures on $\extW$ and hence ergodic. This type of approach cannot be used for Borel windows, and the ergodicity of the measure in this case may look a bit surprising.
The key observation is the recent result that Besicovitch almost periodic measures are generic for ergodic measures on their hull \cite{LSS}, and in this case standard approximations via continuous functions show that all positions of the windows given by points in $\cT$ are generic for the same ergodic measure $\m$.

We complete the paper by reanalyzing the case of compact windows $W$ studied in \cite{BHS,KR}. In this situation we identify a natural set $\YY_{b,\m} \subseteq \extW$ of generic elements with respect to $\m$ such that $m(\YY_{b,\m})=1$, which contains all elements of the form $-s+\oplam(t+W)$ with $(s,t)+\cL \in \cT$. We show that the eigenfunctions are continuous
on $\YY_{b,\m}$ and that there exists a continuous $G$-invariant mapping $\pi : \YY_{b,\m} \to \TT_{red}$ which induces an isometric isomorphism between $(L^2(\extW, \m)$ and $L^2(\TT_{red}, \theta_{\TT_{red}})$. Here, $\TT_{red}$ is the group obtained from $\TT$ via factoring out the periods of $1_{B} \in L^1(H)$. For regular model sets we have $\YY_{b,\m}=\extW=\XX(\oplam(W))$ and $\pi$ is the Schlottmann mapping from \cite{Martin2}.

\medskip

The paper is structured as follows: In Section~\ref{Sect: prel} we review the basic definitions and results needed for the paper. We continue in Section~\ref{sect: min dens open}
by reviewing and extending some results about minimal density model with open window. In Section~\ref{Sect:model sets with norel windows} we prove the main result of the paper in Theorem~\ref{thm:main}. In this Theorem we prove that given a cut and project scheme $(G,H, \cL)$ and a precompact Borel set $B \subseteq H$,
then almost surely all $(s,t) \in \TT$ give a Besicovitch almost periodic measure $\delta_{-s+\oplam(t+B)}$. In particular, for all such $(s,t) \in \TT$, the measure $\delta_{-s+\oplam(t+B)}$ is pure point diffractive, has well defined Fourier--Bohr coefficients which can be calculated via the standard integration over the window $B$ formula, the autocorrelation of $\delta_{-s+\oplam(t+B)}$ is given by the covariogram $c(B)$ of the window, and the intensity of the Bragg peaks is given by the square of the Fourier--Bohr coefficient. Furthermore, there exists an ergodic measure $\m$ on the extended hull $\extB$ with pure point spectrum such that all for $\theta_{\TT}$-almost all $(s,t) \in \TT$ the measure $\delta_{-s+\oplam(t+B)}$ is generic for $\m$ and the Fourier--Bohr coefficients define measurable eigenfunctions in $L^2(\XX, \m)$.

We continue in Section~\ref{Sect: prop m} by studying the properties of the ergodic measure $\m$. We identify two subsets $\YY_{b}$ and $\YY_{b,\m}$ of full measure in $\XX$ which contain all the measures $\delta_{-s+\oplam(t+B)}$ for the generic $(s,t) \in \TT$ given in Theorem~\ref{thm:main} and study their properties. Next, we
identify a subset $\YY_{app} \subseteq \YY_{b, \m}$ and a natural Borel factor mapping $\pi : \YY_{app} \to \TT_{red}$ (where as mentioned before $\TT_{red}$ is the factor of $\TT$ obtained by modding out the periods of the window) and proving that for weak model sets of maximal density, and weak model sets with open window and minimal density, we have
$$
\YY_{b} =\YY_{b,\m} = \YY_{app} \,.
$$
Moreover, these spaces are exactly the set of all elements in the extended hull whose density is exactly the measure of the window. This allows us to recover all known properties for weak model sets of extremal density.

We complete the paper by listing the corresponding results for model sets with open pre-compact windows.

\section{Preliminaries}\label{Sect: prel}

For the entire paper $G$ denotes a second countable locally compact Abelian group (LCAG). We will denote by $\Cu(G)$ the space of uniformly continuous bounded functions on $G$, and by $\Cc(G)$ the subspace of compactly supported continuous functions. The Haar measure on $G$ is denoted by $\theta_G$. For a Borel set $B \subseteq G$ we will often use the notation $|B|:= \theta_G(B)$.

Given a function $f: G \to \CC$ and $t \in G$ we will use the following notations
\begin{displaymath}
f^\dagger(x):= f(-x) \,;\, \tilde{f}(x):=\overline{f(-x)} \,;\, T_tf(x):=f(x-t) \,.
\end{displaymath}

Given two functions $f \in \Cu(G),g \in \Cc(G)$ or $f \in \Cc(G),g \in \Cu(G)$ we define the \textbf{convolution}
\[
f*g(x):= \int_{G} f(x-t)g(t)\dd t = \int_{G} f(t) g(x-t) \dd t \,.
\]

\subsection{Autocorrelation and diffraction}

Here we briefly review the definition of autocorrelation and diffraction measure. For a more detailed overview we recommend \cite{TAO,TAO2,BL,DL03,CRS,CRS2}.

Let us start by reviewing the notion of measures.

\begin{definition} By a \textbf{Radon measure} we mean a linear mapping $\mu : \Cc(G) \to \CC$ with the property that for all compacts $K \subseteq G$ there exists some
$C_K$ such that, for all $f \in \Cc(G)$ with $\supp(f) \subseteq K$ we have
\[
\left| \mu(f) \right| \leq C_K \|f\|_\infty \,.
\]
We will simply refer to a Radon measure as a measure. For properties of measures we refer the reader to \cite{Rud}.

A measure $\mu$ gives rise to a positive measure $\left| \mu \right|$, called the \textbf{total variation measure of $\mu$} with the property that for all $f \in \Cc(G)$ with $f \geq 0$ we have
\begin{displaymath}
\left| \mu \right|(f) = \sup \{ \left| \mu(g) \right| : g \in \Cc(G), |g| \leq f \} \,.
\end{displaymath}
For details see \cite{Ped} or \cite[Appendix]{CRS2}.

A measure $\mu$ is called \textbf{finite} if $\left| \mu \right|(G) <\infty$.
\end{definition}

\smallskip

Given a measure $\mu$ and $t \in G$ we define new measures via
\begin{displaymath}
\mu^\dagger(f):= \mu(f^\dagger) \,;\, \tilde{\mu}(f):=\overline{\mu(\tilde{f})} \,;\, T_t\mu(f):=\mu(T_{-t}f) \,.
\end{displaymath}

For a function $f \in \Cc(G)$ and a measure $\mu$ we define their convolution $f *\mu$ via
\[
f*\mu(x) := \int_{G} f(x-t) \dd \mu(t) \,.
\]

Finally, for a measure $\mu$ and a finite measure $\nu$ we can define the convolution $\mu*\nu$ via (see \cite{ARMA1,MoSt} for details)
\[
\mu*\nu(f) := \int_G \int_G f(s+t) \dd \mu(s) \dd \nu(t) \,.
\]

\smallskip

\begin{definition}
A measure $\mu$ is called \textbf{translation bounded} if
\[
\| \mu \|_{K} := \sup_{t \in G} \left| \mu \right|(t+K) < \infty
\]
for some fixed compact set $K$ with non-empty interior.

We denote the space of translation bounded measures by $\cM^\infty(G)$.
\end{definition}

\begin{remark}
\begin{itemize}
  \item[(a)] If $B$ is any precompact Borel set with non-empty interior, then $\| \, \|_B$ and $\| \, \|_{K}$ are equivalent norms \cite{BM,SS}.
  \item[(b)] A measure is translation bounded if and only if for all $f \in \Cc(G)$ we have $f*\mu \in \Cu(G)$ \cite{ARMA1,MoSt}.
\end{itemize}
\end{remark}

\bigskip

Next, we introduce the autocorrelation and diffraction measures of a translation bounded measure $\omega$. These are defined via use of average sequences, and
we first introduce these.

\begin{definition}
A sequence $(A_k)$  of compact subsets of $G$  called a {\bf van
Hove sequence} if for all compact sets $K \subseteq G$ we have
\[ \lim_{k \to \infty} \frac{\theta_G(\partial^K(
A_k)) } {|A_k|) } = 0\,, \] where the {\bf $K$-boundary} is
defined by:
\[\partial^K(A_k)=\overline{((A_k+K)\backslash A_k)} \cup((\overline{G \backslash A_k} - K) \cap A_k) \,.\]

A van Hove sequence is called \textbf{tempered} if there exists a constant $C$ such that, for all $k$ we have
\[
|\cup_{j=1}^n \left(A_n-A_j\right)| \leq C |A_n | \,,
\]
where $A-B$ denotes the \textbf{Minkowski difference}
\[
A-B=\{ a-b : a \in A, b \in B \} \,.
\]
\end{definition}

A LCAG group admits a tempered van Hove sequence if and only if it is $\sigma$-compact \cite{Martin2,SS1}. Moreover, in this case the van Hove sequence can be chosen to be teppered. Since we rely on van Hove sequences, we will always assume $\sigma$-compactness.

The tempered condition of the van Hove sequence is needed to make sure that the pointwise ergodic theorem holds \cite{Lin}.

\smallskip

\begin{definition} Let $\omega$ be a translation bounded measure and $\cA$ a van Hove sequence. We say that $\omega$ has a well defined \textbf{autocorrelation measure} $\gamma$ with respect to $\cA$ if
\[
\gamma= \lim_k \gamma_k
\]
exists in the vague topology. Here
\[
\gamma_k= \frac{1}{|A_k|} \left( \omega|_{A_k}\right)*\widetilde{\left( \omega|_{A_k}\right)} \,.
\]
\end{definition}

Given a translation bounded measure $\omega$ on a second countable LCAG, the autocorrelation always exists with respect to a subsequence of $\cA$ \cite{BL}.

\bigskip

Next, we introduce the diffraction measure.

Recall first that for some $f \in L^1(G)$ we denote by
\begin{displaymath}
\widehat{f}(\chi):= \int_{G} \overline{\chi(t)} f(t) \dd t
\end{displaymath}
the \textbf{Fourier transform} of $f$ and by
\begin{displaymath}
\widecheck{f}(\chi):= \int_{G} \chi(t) f(t) \dd t = \widehat{f}(\chi^{-1})
\end{displaymath}
the \textbf{inverse Fourier transform} of $f$.

\begin{proposition}\label{Prop FT}\cite{BF,ARMA1,MoSt} Given a translation bounded measure $\omega$ with autocorrelation $\gamma$, there exists a positive measure $\widehat{\gamma}$ on the dual group $\widehat{G}$ such that, for all $f \in \Cc(G)$ we have $\widecheck{f} \in L^2(\widehat{\gamma})$ and
\begin{equation}\label{FT}
\widehat{\gamma}( \left| \widecheck{f} \right|^2) = \gamma(f*\tilde{f}) \,.
\end{equation}
\end{proposition}

\begin{remark} Given a measure $\gamma$ on $G$, if there exists a measure $\widehat{\gamma}$ on $\widehat{G}$ satisfying \eqref{FT}, we say that $\gamma$ is Fourier transformable with Fourier transform $\widehat{\gamma}$.
\end{remark}

\begin{definition} For a measure $\omega \in \cM^\infty(G)$ with autocorrelation $\gamma$, we refer to the measure $\widehat{\gamma}$ from Prop.~\ref{Prop FT} as the \textbf{diffraction of $\omega$} (with respect to $\cA$).
\end{definition}

\subsection{Cut and Project Schemes}

Next, we review the definition of cut and project schemes and properties of weak model sets. For more details we recommend \cite{TAO,TAO2,MOO,Moody2,LR,CRS,NS11,NS19}, just to name a few.

\begin{definition}\label{def:cps} By a \textbf{cut and project scheme} (or simply \textbf{CPS}) we understand a triple $(G,H, \cL)$ consisting of two LCAG $G,H$ and a lattice $\cL \subseteq G \times H$ with the following properties
\begin{itemize}
  \item[(a)] The restriction of the first projection $\pi_G|_{\cL}$ to $\cL$ is one-to-one.
  \item[(b)] $\pi_{H}(\cL)$ is dense in $H$.
\end{itemize}
\end{definition}

We usually assume that $G$ is second countable, as this is needed for diffraction theory, but we will not add this as part of the definition of a CPS.

\smallskip
Given a CPS $(G,H, \cL)$ we denote by $\TT:=(G \times H) /\cL$ and by
\[
L:= \pi_G(\cL) \,,
\]
which is a subgroup of $G$. Then, $\pi_{G}|_{\cL} : \cL \to L$ becomes an isomorphism, and hence we can define a mapping
\[
\star : L \to H \, ;\, \star := \pi_{H} \circ \left( \pi_{G}|_{\cL} \right)^{-1} \,,
\]
called the \textbf{star mapping} of the CPS.

This mapping allows us re-parameterise $\cL$ as
\[
\cL= \{ (x,x^\star) : x \in L \} \,.
\]

\medskip
Consider a CPS $(G,H, \cL)$ and a set $W \subseteq H$. We define
\[
\oplam(W)= \{ x \in L : x^\star \in W \} \,.
\]
Note here that if $W$ is precompact, then $\oplam(W)$ is uniformly discrete and if $W^\circ \neq \emptyset$ then $\oplam(W)$ is relatively dense \cite{MOO}.

\medskip
Also, given some $h : H \to \CC$ we can define a formal sum
\[
\omega_h:= \sum_{x \in L} h(x^\star) \delta_x \,.
\]
If $h$ is compactly supported and bounded, then $\omega_h$ is a measure. There are many other conditions on $h$ which ensure that $\omega_h$ is a measure, see for example \cite{LR,CRS,NS19,NS20}.

\smallskip

Note here that for all $s \in L$ we have:
\begin{align}
 \oplam(s^\star+W)  &=\{x \in L : x^\star \in s^\star +W\} =\{x \in L : (x-s)^\star  \in W\}= s+\oplam(W) \label{eq:translate oplam(W)}\\
  T_s\omega_h &=\sum_{x \in L} h(x^\star) \delta_{s+x} =\sum_{y \in L} h(y^\star-s^\star) \delta_{y}= \omega_{T_{s^\star}h} \,. \label{eq:translate omegah}
\end{align}
We will use these formulas often.

\bigskip

Let us review next the concept of dual CPS.  Given a CPS $(G,H, \cL)$ we can define
\[
\cL^0:= \{ (\chi, \psi) \in \widehat{G} \times \widehat{H} : \chi(x)\psi(x^\star)=1 \forall x \in L \} \,.
\]
Then, $(\widehat{G}, \widehat{H}, \cL^0)$ is a CPS\cite{MOO}, called the \textbf{CPS dual to $(G, H, \cL)$}. This dual CPS plays the central role in the study of diffraction of model sets.

\bigskip

Given a uniformly discrete set $\Lambda \subseteq G$ and a van Hove sequence $\cA$ we define
\begin{align*}
\ldens_{\cA}(\Lambda) &:= \liminf_k \frac{\card(\Lambda \cap A_k)}{|A_k|} \\
\udens_{\cA}(\Lambda) &:= \limsup_k \frac{\card(\Lambda \cap A_k)}{|A_k|} \,.
\end{align*}

If the following limit exists, or equivalently if $\ldens_{\cA}(\Lambda)=\udens_{\cA}(\Lambda)$, we define
\[
\dens_{\cA}(\Lambda)= \lim_k \frac{\card(\Lambda \cap A_k)}{|A_k|} \,.
\]

\medskip

Let us next recall the following result of \cite{HR}.

\begin{theorem}\cite{HR} Let $(G,H, \cL)$ be a CPS and $W \subseteq H$ be pre-compact. Then, for all van Hove sequences $\cA$ we have
\begin{displaymath}
  \dens(\cL) \theta(W^\circ) \leq \ldens(\oplam(W)) \leq \udens(\oplam(W)) \leq \dens(\cL) \theta(\overline{W}) \,.
\end{displaymath}
\end{theorem}

We will use the following definitions.

\begin{definition} Let $(G,H, \cL)$ be a CPS and let $W \subseteq H$ be precompact.
\begin{itemize}
\item[(a)]If $W$ is compact, we call $\oplam(W)$ a \textbf{weak model set}.
\item[(b)] If $W$ is open, we call $\oplam(W)$ a \textbf{open model set}.
\item[(c)] We say that $\oplam(W)$ has \textbf{maximal density with respect to $\cA$} if
\[
 \dens_{\cA}(\oplam(W)) = \dens(\cL) \theta(\overline{W}) \,.
 \]
 \item[(d)] We say that $\oplam(W)$ has \textbf{minimal density with respect to $\cA$} if
\[
 \dens_{\cA}(\oplam(W)) = \dens(\cL) \theta(W^\circ) \,.
 \]
\end{itemize}
 \end{definition}

\begin{remark}
\begin{itemize}
  \item[(a)] If $\oplam(W)$ has maximal density with respect to $\cA$ then $\oplam(\overline{W})$ is a maximal density weak model set. Such models have been studied in \cite{BHS,KR,KR2,Kel1}.
  \item[(b)] If $\oplam(W)$ has minimal density with respect to $\cA$ then $\oplam(W^\circ)$ is a minimal density open model set. We derive the theory of such models in the next section, from the theory of maximal density weak model sets.
\end{itemize}
\end{remark}

\bigskip
Let us complete this subsection by recalling the following result from \cite{Moody}. This result will play an essential role in the remaining of the paper.

\begin{theorem}\cite[Thm.~1]{Moody}\label{thm:moody} Let $(G,H, \cL)$ be CPS and let $W \subseteq H$ a measurable pre-compact set, and let $\cA$ be a tempered van Hove sequence in $G$. Then, for almost all $(s,t)+\cL \in (G \times H)/ \cL =: \TT$ we have\footnote{Note that in \cite{Moody} the Haar measure on $H$ is renormalized such that $\dens(\cL)=1$.}
\begin{equation}\label{eq: def generic}
\dens_{\cA}(-s+\oplam(t+W))= \lim_{k} \frac{1}{|A_k|} \card( \left(-s+\oplam(t+W)\right) \cap A_k ) = \dens(\cL) \theta_{H}(W) \,.
\end{equation}
\end{theorem}

\begin{remark}  Theorem~\ref{thm:moody} says that given a CPS an open set $U \subseteq H$ and a compact set $W$, for almost all $(s,t)+\cL \in \TT$ the set
$\oplam(t+U)$ is a minimal density open model set and $\oplam(t+W)$ is a maximal density weak model set.
\end{remark}

\subsection{Almost periodic functions and measures}

In this subsection we briefly review the concepts of almost periodicity which we will use in the paper. For this entire subsection $\cA$ a fixed van Hove sequence.

First let us recall that a function $f \in \Cu(G)$ is called \textbf{Bohr almost periodic} if, for each $\eps>0$ the set
\[
P_\eps(f):= \{ t \in G : \| T_tf-f\|_\infty < \eps \}
\]
is relatively dense. A measure $\mu$ is called \textbf{strong almost periodic} if for all $f \in \Cc(G)$ the function $f*\mu$ is Bohr almost periodic. We denote the spaces of Bohr almost periodic functions by $SAP(G)$, and the space of strong almost periodic measures by $\SAP(G)$. For the basic properties of these measures we refer the reader to \cite{ARMA,MoSt}.

\smallskip

Next, let us review the larger classes of Besicovitch almost periodic functions and measures. For $f \in L^1_{loc}(G)$ we define
\[
\|f\|_{b,1, \cA} = \limsup_k \frac{1}{|A_k|} \int_{A_k} |f(t)| \dd t \,.
\]
For a measure $\mu \in \cM^\infty$ we will denote by
\[
\uM_{\cA}(|\mu|):= \limsup_k \frac{\left| \mu \right(A_k)}{|A_k|} \,.
\]

Note here that for $f \in \Cc(G)$ and $\mu \in \cM^\infty(G)$ we have
\begin{displaymath}
\| f*\mu \|_{b,1,\cA}\leq \uM_{\cA}(|\mu|) \|f \|_1 \,.
\end{displaymath}

Now, let us recall the concept of Besicovitch almost periodicity for functions and measures. For more details we refer the reader to \cite{LSS,LSS2}.

\begin{definition} A function $f \in L^1_{loc}(G)$ is called \textbf{Besicovitch almost periodic with respect to $\cA$} if, for each $\eps >0$ there exists a trigonometric polynomial $P$ such that
\[
\| f-P \|_{b,1,\cA} < \eps \,.
\]
We denote the space of Besicovitch almost periodic functions by $Bap_{\cA}(G)$.

A measure $\mu \in \cM^\infty(G)$ is called \textbf{Besicovitch almost periodic with respect to $\cA$} if for all $f \in \Cc(G)$ we have $\mu*f \in Bap_{\cA}(G)$.
The space of Besicovitch almost periodic measures is denote by $\Bap_{\cA}(G)$.
\end{definition}

\begin{remark} On the space $Bap_{\cA}(G)$ we have \cite{LSS}
\[
\| f\|_{b,1, \cA}=  \lim_k \frac{1}{|A_k|} \int_{A_k} |f(t)| \dd t \,.
\]

\end{remark}

\medskip

Let us now recall the following result about Fourier--Bohr coefficients.

\begin{proposition}\label{prop: FB for besicovitch} \cite{LSS}
\begin{itemize}
  \item[(a)] Let $f \in Bap_{\cA}(G)$ and $\chi \in \widehat{G}$. Then, the \textbf{ Fourier--Bohr coefficient}
  \[
  a_{\chi}^\cA(f):= \lim_k \frac{1}{|A_k|} \int_{A_k} \overline{\chi(t)}f(t) \dd t
  \]
  exists.
  \item[(b)] Let $\mu \in \Bap_{\cA}(G)$ and $\chi \in \widehat{G}$. Then, the \textbf{ Fourier--Bohr coefficient}
  \[
  a_{\chi}^\cA(\mu):= \lim_k \frac{1}{|A_k|} \int_{A_k} \overline{\chi(t)} \dd \mu(t)
  \]
  exists.
\end{itemize}
\end{proposition}

\medskip

In this paper we will be interested in pointsets $\Lambda$ for which $\delta_{\Lambda} \in \Bap_{\cA}(G)$. Because of this we introduce the following notation:
\[
\Bap_{\cA, ps}(G):= \{ \Lambda \subseteq G : \delta_{\Lambda} \in \Bap_{\cA}(G) \} \,.
\]

\subsection{Dynamical systems}
We complete this section by reviewing few basic facts about Dynamical Systems and ergodic measures.

\smallskip

By a \textbf{(topological) dynamical system} we mean a compact topological space $X$ together with a continuous group action $\alpha : G \times X \to X$.

A measure $m$ on $X$ is called $G$-invariant if for all $f \in C(X)$ and all $t \in G$ we have $m(T_tf)=m(f)$. Here $T_t(f)(x)= f(\alpha(-t,x))$.
A probability $G$-invariant measure $m$ is called \textbf{ergodic} if all measurable sets $B \subseteq X$ which are $G$-invariant satisfy $m(B) \in \{0,1\}$.

Ergodic measures always exist on topological dynamical systems. Given an ergodic measure $m$ on a topological dynamical system, we will refer to the triple
$(X, G, m)$ as an ergodic dynamical system.

\smallskip

Of special interest to us will be dynamical systems of translation bounded measures. Let us start by recalling the following result of \cite{SS1} (compare \cite{BL}).

\begin{lemma}\cite{SS1} Let $\XX \subseteq \cM^\infty(G)$ be any $G$ invariant vaguely closed set. Then, $\XX$ is vaguely compact if and only if it there exists some $C$ and a compact set $K \subseteq G$ with non-empty interior such that
\[
\XX \subseteq \{ \mu \in \cM^\infty(G): \| \mu \|_{K} \leq C \} =: \cM_{C,K}(G) \,.
\]

Moreover, in this case the vague topology on $\XX$ is metrisable.
\end{lemma}

This leads to the following natural definition
\begin{definition} \cite{BL} A pair $(\XX, G)$ is called a \textbf{dynamical system on
the translation bounded measures} on $G$ (TMDS) if $\XX$ is a vaguely compact set of translation bounded measures on $G$.
\end{definition}

\bigskip

Next, let us review the following definition, which is inspired by the general Birkhoff ergodic theorem for continuous group actions \cite{Cha,Tem}.

\begin{definition} Let $(X,G,m)$ be an ergodic dynamical system and let $\cA$ be a van Hove sequence. An element $x \in X$ is called \textbf{generic} for $m$ if
\begin{displaymath}
\lim_{m} \frac{1}{|A_k|} \int_{A_k} f(T_t x) \dd t = \int_{X} f(y) \dd m(y)
\end{displaymath}
holds for all $f \in C(X)$.
\end{definition}

We will need the following simple characterisations of generic elements. This has been used implicitly in  \cite{BL,LSS,LS,LSS2} just to name a few. Since the result is well known and trivial to prove we skip the proof.

\begin{proposition}\label{prop:generic char2} Let $(X, G, m)$ be an ergodic dynamical system. Then, $x \in X$ is generic for $m$ if and only if the set $f \in C(X)$ for which
\[
\int_{X} f(y) \dd m(y)= \lim_m \frac{1}{|A_m|} \int_{A_m} f(T_tx) \dd t
\]
is dense in $C(X)$.
\end{proposition}
%
%

%



As an immediate consequence we get:

\begin{proposition}\label{prop:generic char} Let $(\XX, G, m)$ be an ergodic TMDS consisting of real valued measures. Then, $\mu \in \XX$ is generic for $m$ if and only if, for all $n \in \NN$ and $\varphi_1,..., \varphi_n \in \Cc(G)$ we have
\begin{equation}\label{eq:generic}
\lim_{m} \frac{1}{|A_m|} \int_{A_m} \mu*\varphi_1(t) \mu*\varphi_2(t) \cdot \ldots \cdot \mu*\varphi_n(t) \dd t = \int_{\XX(B)} f_{\varphi_1}(\omega) f_{\varphi_2}(\omega)  \cdot \ldots \cdot f_{\varphi_n}(\omega) \dd m(\omega) \,.
\end{equation}
\end{proposition}
\begin{proof}
$\Longrightarrow$: is obvious.

$\Longleftarrow$: Since all measures are real valued, we have $\overline{f_{\varphi}(\omega)}=f_{\overline{\varphi}}(\omega)$ for all $\varphi \in \Cc(G)$ and $\omega \in \XX(B)$.

Therefore, the algebra $\AAA$ generated by $\{ f_ \varphi : \varphi  \in \Cc(G) \} \cup \{1_{\XX} \}$ is a complex algebra which clearly separates the points of $\XX$. Stone--Weiertra{\ss} theorem implies then that $\AAA$ is dense in $C(\XX)$.

Since \eqref{eq: def generic} trivially holds for $1_{\XX}$, it also holds for all $f \in \AAA$ by \eqref{eq:generic}. The claim follows now by the density of $\AAA$ in $C(\XX)$ and Prop~\ref{prop:generic char2}.

\end{proof}

Finally, we will need the notion of a Borel factor.

\begin{definition} Let $(X,G, m)$ and $(X', G, m')$ be two ergodic dynamical system. We say that $(X', G, m')$ is a \textbf{Borel factor} of $(X,G,m)$ if there exists sets $Y \subseteq X$ and $Y ' \subseteq Y$ of full measure, and a $G$-mapping $f: X' \to Y'$ such that $m'$ is the push forward of $m$ via $f$, that is
\[
\int_{X'} g(t) \dd m'(t) = \int_{X} g(f(t)) \dd m(t) \qquad \forall g \in C(X) \,.
\]
\end{definition}

\subsection{The rubber topology and the extended hull}

Consider the set $\mathcal{UD}(G)$ of uniformly discrete subsets of $G$. We have a natural embedding \cite{BL} $i : \mathcal{UD}(G) \hookrightarrow \cM^\infty(G)$ defined by
\[
i(\Lambda):= \delta_{\Lambda}= \sum_{x \in \Lambda} \delta_{x} \,.
\]

The topology induced by this embedding on $\mathcal{UD}(G)$ from the vague topology on $\cM^\infty(G)$ is called the \textbf{local rubber topology}. One can define equivalently this topology by saying that two sets $\Lambda, \Gamma \in \mathcal{UD}(G)$ are closed if they "almost" agree on large compact sets (see \cite{BL} for details).

\smallskip
Let us recall next the following result\cite{BL}.
\begin{lemma}\cite[Prop~4]{BL} Let $0 \in V \subseteq G$ be an open neighbourhood, and let $\mathcal{D}_{V}(G)$ be the set of $V$-uniformly discrete subsets of $G$. Then
$\mathcal{D}_{V}(G)$ is compact in the local rubber topology and $i(\mathcal{D}_{V}(G))$ is vaguely compact in $\cM^\infty(G)$.
\end{lemma}

\begin{corollary} Let $\Lambda$ be an uniformly discrete pointset. Then $i$ induces a homeomorphism
$$
i: \XX(\Lambda) \to \XX(\delta_{\Lambda}) \,.
$$
\end{corollary}

It is worth noting here that $i(\mathcal{UD}(G))$ is not vaguely closed in $\cM^\infty(G)$ which can create issues when one studies pointsets which are uniformly discrete but not equi uniformly discrete.

\subsection{Dynamical systems from CPS}

Consider a CPS $(G,H,\cL)$ and a precompact set $B \subseteq H$.

The \textbf{hull $\XX(\oplam(B))$ of $\oplam(B)$} is defined as
\[
\XX(B):= \overline{\{ s+\oplam(B) : s \in G \}} \,,
\]
where the closure is taken in the local rubber topology.

Following the notation of \cite{KR}, \textbf{extended hull $\extB$ of $B$} is defined as
\[
\extB:= \overline{\{ s+\oplam(t+B) : s \in G, t \in H \}} \,,
\]
where the closure is taken in the local rubber topology.

For regular model sets we have $\extB=\XX(\oplam(B))$, but this is in general not true. Indeed, consider the CPS $(\RR, H, \cL)$ and the window $W$ of the visible points of the lattice \cite{BHS}. Then, there exists positions $t \in H$ such that $(t+W) \cap L^\star =\emptyset$. Pick one such position $t$ and set $B=t+W$. Then, $\XX(\oplam(B))= \{ \emptyset \}$ while the extended hull $\extB$ contains the visible points of the lattice.

\medskip

We complete this section by reviewing the following results, which we will use in the paper.

\begin{theorem}\label{thm:omegah in sap}\cite{LR,LLRSS} Let $(G,H,\cL)$ be a CPS and $g \in \Cc(H)$. Then $\omega_g \in \SAP(G)$. Moreover, the  hull $\XX(\omega_{g})$ is a compact Abelian group.
\end{theorem}

Let us note the following simple lemma.

\begin{lemma}\label{L1} Let $(G,H,\cL)$ be a CPS and $g \in \Cc(H)$ and let $m$ denote the unique ergodic measure on $\XX(\omega_g)$. Then,
\begin{itemize}
  \item[(a)] For all $t \in H$ we have $\omega_{T_tg} \in \XX(\omega_g)$.
  \item[(b)] For all $f \in C(\XX(\omega_g))$ and all $(s,t) \in G \times H$ we have
  \[
  m(f)= \lim_k \frac{1}{|A_k|} \int_{A_k} f( T_{s+x} \omega_{T_tg}) \dd x \,.
  \]
\end{itemize}
\end{lemma}
\begin{proof}

\textbf{(a)} Since $L^\star$ is dense in $H$ we can find a precompact neighbourhood $t \in V$ and a net $t_{\alpha}$ with $t_\alpha \in V$ which converges in $H$ to $t$. Then, using the uniform continuity of $g$ we get that
\begin{equation}\label{eq:99}
\lim_\alpha \| T_{t_\alpha}g - T_{t}g \|_\infty =0 \,.
\end{equation}
Next, all these functions are supported in the compact set $\overline{V}+\supp(g)$. Therefore, all the measures $\omega_{ T_{t_\alpha}g},  \omega_{ T_{t}g}$ are supported inside the uniformly discrete set $\oplam(\overline{V}+\supp(g))$. Moreover, \eqref{eq:99} implies that for all $x \in \oplam(\overline{V}+\supp(g))$ we have
\[
\omega_{ T_{t}g}( \{x \})= \lim_\alpha \omega_{ T_{t_\alpha}g}( \{x \}) \,.
\]
This combined with the equi uniform discreteness of their support implies that $\omega_{ T_{t_\alpha}g}$ converges vaguely to $\omega_{ T_{t}g}$.

Finally, since $t_\alpha \in L^\star$ there exists some $s_\alpha\in L$ such that $t_\alpha = s_\alpha^\star$. Then, by \eqref{eq:translate omegah} we have
\[
\omega_{ T_{t_\alpha}g}=T_{s_\alpha} \omega_g \in \XX(\omega_g) \,.
\]
The claim follows.

\medskip

\textbf{(b)} follows immediately from (a) and the unique ergodic theorem.

\end{proof}

\section{Minimal density open model sets}\label{sect: min dens open}

In this section we review the theory of minimal density open model sets from the corresponding theory of maximal density weak model sets from \cite{BHS}, and derive few new results.

\begin{lemma}\label{lemma 1} Let $(G, H, \cL)$ be a CPS and let $U \subseteq H$ be a precompact open set and $K$ a regular window such that $U \subseteq K$.
Then,
\begin{itemize}
  \item[(i)] $\oplam(K)=\oplam(K \backslash U) \bigcupdot \oplam(U)$.
  \item[(ii)] $\oplam(K \backslash U)$ has maximal density with respect to $\cA$ if and only if $\oplam(U)$ has minimal density with respect to $\cA$.
\end{itemize}
\end{lemma}
\begin{proof}
\textbf{(a)} is obvious.

\textbf{(b)}
$\Longrightarrow$:

Since $K$ is regular model set and $\oplam(K \backslash U)$ is maximal density, we have
\begin{align*}
\lim_k \frac{1}{|A_k|} \card(\oplam(K)\cap A_k)&=\dens(\cL) \theta_{H}(K) \,,\\
\lim_k \frac{1}{|A_k|} \card(\oplam(K\backslash U)\cap A_k)&=\dens(\cL) \theta_{H}(K\backslash U) \,. \\
\end{align*}
Now by (a)
\begin{equation}\label{EQ1}
\card(\oplam(K\backslash U)\cap A_k) + \card(\oplam(U)\cap A_k)=\card(\oplam(K)\cap A_k) \,.
\end{equation}
Since $U \subseteq K$ we also have
\begin{equation}\label{eq2}
\theta_H(K)=\theta_H(K\backslash U)+\theta_H(U) \,.
\end{equation}

This immediately gives
$$
\lim_n \frac{1}{|A_k|} \card(\oplam(U)\cap A_k)=\dens(\cL) \theta_{H}(U) \,.
$$

$\Longleftarrow$: is similar.

Since $K$ is regular model set and $\oplam(U)$ is minimal density, we have
\begin{align*}
\lim_k \frac{1}{|A_k|} \card(\oplam(K)\cap A_k)&=\dens(\cL) \theta_{H}(K) \,, \\
\lim_k \frac{1}{|A_k|} \card(\oplam(U)\cap A_k)&=\dens(\cL) \theta_{H}(U) \,.\\
\end{align*}
Now by \eqref{EQ1} and \eqref{eq2} we get
$$
\lim_n \frac{1}{|A_k|} \card(\oplam(K\backslash U)\cap A_k)=\dens(\cL) \theta_{H}(K\backslash U) \,.
$$
\end{proof}

As consequence we get(compare \cite{BHS}).

\begin{theorem}\label{thm: min density model sets} Let $\cA$ be a van Hove sequence, $(G,H, \cL)$ be a CPS and $U \subseteq H$ be open precompact set such that $\oplam(U)$ is minimal density open model set with respect to $\cA$. Then,
\begin{itemize}
  \item[(a)] $\oplam(U) \in \Bap_{\cA,ps}(G)$.
  \item[(b)] The autocorrelation $\gamma$ of $\oplam(U)$ exists and satisfies
  \[
  \gamma=\dens(\cL) \omega_{c(U)} = \dens( \cL) \sum_{(x,x^\star) \in \cL} \theta_H(U \cap (x^\star +U)) \delta_x \,.
  \]
  \item[(c)] The diffraction $\widehat{\gamma}$ is given by
\[
\widehat{\gamma}= \dens(\cL)^2 \sum_{(\chi,\chi^\star) \in \cL^0} \left| \int_{U} \chi^\star(t) \dd t \right|^2 \delta_\chi \,.
\]
  \item[(d)] The Fourier--Bohr coefficients
  \[
a_\chi^\cA( \delta_{\oplam(U)})= \lim_k \frac{1}{|A_k|} \sum_{x \in \oplam(U) \cap A_k} \overline{\chi(x)}
  \]
  exists and satisfy
  \[
a_\chi^\cA( \delta_{\oplam(U)})= \left\{
\begin{array}{cc}
\dens(\cL) \widecheck{1_U}(\chi^\star)=\dens(\cL) \int_{U} \chi^\star(t) \dd t  & \mbox{ if } \chi \in \pi_{\hat{G}}(\cL^0) \\
 0 & \mbox{ otherwise }
\end{array}
\right. \,.
  \]
  \item[(e)] $\delta_{\oplam(U)}$  satisfies the Consistent Phase Property
\[
\widehat{\gamma}(\{ \chi \}) = \left| a_\chi^\cA( \delta_{\oplam(U)}) \right|^2 \qquad \forall \chi \in \widehat{G} \,.
\]
  \item[(f)] There exist a $G$-invariant ergodic measure $m$ on $\XX=\XX(\oplam(U))$ such that $\oplam(U)$ is generic for $m$.
  \item[(g)] $(\XX, G, m)$ has pure point dynamical spectrum.
  \item[(h)] $\gamma$ is the autocorrelation of $(\XX, G, m)$.
\end{itemize}
\end{theorem}
\begin{proof}
\textbf{(a)} Pick some regular set $K$ such that $U \subseteq K$. Such a set exists by \cite{HR}.

Then, we have
\begin{displaymath}
\delta_{\oplam(U)}= \delta_{\oplam(K)} - \delta_{\oplam(K \backslash U)} \,.
\end{displaymath}

Since $\oplam(K)$ is a regular model set, the measure $\delta_{\oplam(K)}$ is Weyl almost periodic \cite[Cor.~4.26]{LSS} and hence Besicovitch almost periodic with respect to $\cA$ \cite[Prop.~4.3]{LSS}.

Next, since $\oplam(U)$ is a minimal density model set with respect to $\cA$, by Lemma~\ref{lemma 1} $\oplam(K \backslash U)$ is a maximal density model set with respect to $\cA$, and hence $\delta_{\oplam(K \backslash U)} \in \Bap_{\cA}(G)$ \cite[Prop.~3.39]{LSS}.

Therefore,
$$\delta_{\oplam(U)}= \delta_{\oplam(K)} - \delta_{\oplam(K \backslash U)} \in \Bap_{\cA}(G) \,.$$

\textbf{(b),(c), (d)} follow from \cite[Thm.~9]{BHS}.

\textbf{(e)} follow from \cite[Thm.~9]{BHS} or \cite[Thm.~3.36]{LSS}.

\textbf{(f)} follows from \cite[Thm.~6.13]{LSS} or \cite[Thm.~3.4 and Prop.~6.1]{LSS2}.

\textbf{(g), (h)} follows from \cite[Thm.~6.13]{LSS}
\end{proof}

\section{Model sets with Borel windows}\label{Sect:model sets with norel windows}

In this section we study model sets with precompact Borel windows. We show that almost surely all positions of the windows yield Besicovitch almost periodic point sets. This allows us calculate the autocorrelation, diffraction and Fourier--Bohr coefficients of all these point sets.

\begin{theorem}\label{thm:main} Let $(G,H, \cL)$ be a CPS and $B \subseteq H$ be Borel precompact set. Then, there exists a set $\cT \subseteq \TT$ of full measure with the following properties:
\begin{itemize}
  \item[(a)]  For all $(s,t)+\cL \in \cT$ we have $-s+\oplam(t+B) \in \Bap_{\cA,ps}(G)$.
    \item[(b)] For each $(s,t)+\cL \in \cT$ we have
  \[
  \lim_k \frac{1}{|A_k|} \card( (-s+\oplam(t+B)) \cap A_k ) = \dens(\cL) \theta_H(B) \,.
  \]
  \item[(c)] For each $(s,t)+\cL \in \cT$ the autocorrelation $\gamma$ of $-s+\oplam(t+B)$ exists, is independent of the choice of $(s,t)+\cL \in \cT$, and satisfies
  \[
  \gamma=\dens(\cL) \omega_{c(B)} = \dens( \cL) \sum_{(x,x^\star) \in \cL} \theta_H(B \cap (x^\star +B)) \delta_x \,.
  \]
  \item[(d)] The diffraction $\widehat{\gamma}$ is given by
\[
\widehat{\gamma}= \dens(\cL)^2 \sum_{(\chi,\chi^\star) \in \cL^0} \left| \int_{B} \chi^\star(t) \dd t \right|^2 \delta_\chi \,,
\]
  \item[(e)] For each $(s,t)+\cL \in \cT$, the Fourier--Bohr coefficients
  \[
a_\chi^\cA(-s+\oplam(t+B))=\chi(s) \lim_k \frac{1}{|A_k|} \sum_{x \in \oplam(t+B) \cap A_k} \overline{\chi(x)}
  \]
  exist and satisfy
  \[
a_\chi^\cA(-s+\oplam(t+B))= \left\{
\begin{array}{cc}
\dens(\cL) (\chi(s)\chi^\star(t)  ) \int_{B} \chi^\star(r) \dd r  & \mbox{ if } \chi \in \pi_{\hat{G}}(\cL^0) \\
 0 & \mbox{ otherwise }
\end{array}
\right. \,.
  \]
  \item[(f)]For all $(s,t)+\cL \in \cT$ the set $-s+\oplam(t+B)$  satisfies the Consistent Phase Property
\[
\widehat{\gamma}(\{ \chi \}) = \left| a_\chi^\cA( -s+\oplam(t+B)) \right|^2 \qquad \forall \chi \in \widehat{G} \,.
\]
  \item[(g)] There exist a $G$-invariant ergodic measure $\m$ on the extended hull $\extB$ such that, for all $(s,t) +\cL \in \cT$ the set $-s+\oplam(t+B)$ is generic for $\m$.
  \item[(h)] $(\extB, G, \m)$ has pure point dynamical spectrum generated by $\{ \chi : \widehat{\gamma}(\{\chi \}) \neq 0 \}$.
  \item[(i)] $\gamma$ is the autocorrelation of $(\extB, G, \m)$.
  \item[(j)] $\m( \extB \cap \Bap_{\cA,ps}(G))=1$.
  \item[(k)] For each $\chi$ with $\widehat{\gamma}(\{ \chi \}) \neq 0$ the function
  \[
\extB \cap \Bap_{\cA,ps}(G) \ni \omega \to a_{\chi}^{\cA} ( \omega)
  \]
  is a measurable eigenfunction for $\extB$.
\end{itemize}
\end{theorem}
\begin{proof}

Fix first an open set $U$ such that $U$ is precompact and $B \subseteq U$.

Since $B$ is precompact, $\theta_H(B) < \infty$. Then, for each $n$, by \cite[Thm.~2.14 (c) and (d)]{Rud}, there exists some compact set $K_n$ and open precompact set $U_n \subseteq U$ such that $K_n \subseteq B \subseteq U_n$ and
\[
\theta_{H} (U_n \backslash K_n) < \frac{1}{n} \,.
\]

By Theorem~\ref{thm:moody}, there exists sets $X_n, Y_n \subseteq \TT$ of full measure such that
\begin{align*}
\lim_{k} \frac{1}{|A_k|} \card( \left(-s+\oplam(t+K_n)\right) \cap A_k) &= \dens(\cL) \theta_{H}(K_n) \qquad \forall (s,t) +\cL \in X_n \,,  \\
\lim_{k} \frac{1}{|A_k|} \card( \left(-s+\oplam(t+U_n)\right) \cap A_k ) &= \dens(\cL) \theta_{H}(U_n) \qquad \forall (s,t) +\cL \in Y_n \,. \\
\end{align*}

Define
\[
\cT:= \bigcap_n (X_n \cap Y_n) \,.
\]
Then, $\cT$ has full measure in $\TT$ and for all $(s,t)+\cL \in \cT$ we have
\begin{align*}
\lim_{k} \frac{1}{|A_k|} \card( \left(-s+\oplam(t+K_n)\right) \cap A_k ) &= \dens(\cL) \theta_{H}(K_n)  \,, \\
\lim_{k} \frac{1}{|A_k|} \card( \left(-s+\oplam(t+U_n)\right) \cap A_k ) &= \dens(\cL) \theta_{H}(U_n)  \,.\\
\end{align*}

We claim that this $\cT$ satisfies the given conditions.

\textbf{(a)} Let $(s,t) +\cL \in \cT$.

Let $\varphi \in \Cc(G)$. Then, for all $n$ we have
\begin{align*}
&\bigl\|\delta_{-s+\oplam(t+U_n)}*\varphi(t) - \delta_{ -s+\oplam(t+B)}*\varphi(t) \bigr\|_{b,1,\cA} \\
 &\leq  \| \varphi \|_1 \uM_{\cA} \left|\delta_{-s+\oplam(t+U_n)} - \delta_{ -s+\oplam(t+B)}\right| \,.\\
\end{align*}

Moreover, since $K_n \leq B \leq U_n$, the measure $\delta_{-s+\oplam(t+U_n)} - \delta_{ -s+\oplam(t+B)}$ is positive and
\begin{align*}
&\uM_{\cA}( \left|\delta_{-s+\oplam(t+U_n)} - \delta_{ -s+\oplam(t+B)}\right|) \\
&=\limsup_{k} \frac{1}{|A_k|} \delta_{-s+\oplam(t+U_n)} - \delta_{ -s+\oplam(t+B)}(A_k) \\
&\leq \lim_{k} \frac{1}{|A_k|} \delta_{-s+\oplam(t+U_n)} - \delta_{ -s+\oplam(t+K_n)}(A_k) \\
&= \dens(\cL) \left( \theta_H(U_n)- \theta_H(K_n) \right) < \frac{\dens(\cL)}{n} \,.
\end{align*}

Therefore,
\begin{displaymath}
\bigl\|\delta_{-s+\oplam(t+U_n)}*\varphi(t) - \delta_{ -s+\oplam(t+B)}*\varphi(t)\bigr\|_{b,1,\cA}  \leq  \frac{\dens(\cL) \cdot \| \varphi \|_1}{n}
\end{displaymath}

This shows that $\| \delta_{-s+\oplam(t+U_n)}*\varphi - \delta_{ -s+\oplam(t+B)}*\varphi \|_{b,1, \cA} \to 0$.

Now,
\[
\lim_{k} \frac{1}{|A_k|} \card( \left(-s+\oplam(t+U_n)\right) \cap A_k ) = \dens(\cL) \theta_{H}(U_n) \,,
\]
implies
\[
\lim_{k} \frac{1}{|A_k|} \card( \left(\oplam(t+U_n)\right) \cap A_k ) = \dens(\cL) \theta_{H}(U_n) \,.
\]
Since $U_n$ is open, it follows that $\oplam(t+U_n)$ is minimal density model set, and hence $\delta_{\oplam(t+U_n)} \in \Bap_{\cA}(G)$ by Theorem~\ref{thm: min density model sets}. Since $\delta_{\oplam(t+U_n)}  \in \cM^\infty(G)$ we get $\delta_{-s+\oplam(t+U_n)} \in \Bap_{\cA}(G)$ and hence $\delta_{-s+\oplam(t+U_n)}*\varphi \in Bap_{\cA}(G)$.

Since $\| \delta_{-s+\oplam(t+U_n)}*\varphi - \delta_{ -s+\oplam(t+B)}*\varphi \|_{b,1, \cA} \to 0, \delta_{-s+\oplam(t+U_n)}*\varphi \in Bap_{\cA}(G)$ and $Bap_{\cA}(G)$ is complete by \cite[Thm.~3.10]{LSS}, we get that $\delta_{ -s+\oplam(t+B)}*\varphi  \in Bap_{\cA}(G)$.

This proves (a).

\medskip

\textbf{(b)} Let $(s,t) +\cL \in \cT$. Since $\delta_{-s+\oplam(t+B)} \in \Bap_{\cA}(G)$, the limit
 \[
L:=  \lim_k \frac{1}{|A_k|} \card( (-s+\oplam(t+B)) \cap A_k )
 \]
exists by \cite[Thm.~3.36(b)]{LSS} applied to the trivial character.

For simplicity, we denote by
\[
\cM_{\cA}(\mu) :=  \lim_k \frac{1}{|A_k|} \mu(A_k )
\]
whenever when the limit exists. Now, since $-s+\oplam(t+K_n) \subseteq -s+\oplam(t+B)) \subseteq -s+\oplam(t+U_n)$ we have for all $n$.
\[
\dens(\cL) \theta_H(K_n)= \cM_{\cA} ( \delta_{-s+\oplam(t+K_n)}) \leq L \leq  \cM_{\cA} ( \delta_{-s+\oplam(t+U_n)}) = \dens(\cL) \theta_H(U_n) \,.
\]

Therefore, for all $n$ we have
\begin{align*}
\dens(\cL) \theta_H(K_n) &\leq L \leq  \dens(\cL) \theta_H(U_n) \\
\dens(\cL) \theta_H(K_n) &\leq \dens(\cL) \theta_H(B) \leq  \dens(\cL) \theta_H(U_n) \\
\dens(\cL) \theta_H(U_n)&-\dens(\cL) \theta_H(K_n) < \frac{1}{n}
\end{align*}

It follows that $\left| \dens(\cL) \theta_H(B) - L \right| <\frac{1}{n}$  for all $n$ and hence
\begin{displaymath}
\lim_k \frac{1}{|A_k|} \card( (-s+\oplam(t+B)) \cap A_k ) = L = \dens(\cL) \theta_H(B) \,.
\end{displaymath}

\medskip
\textbf{(c)}

For each $(s,t) +\cL \in \cT$ we have $\delta_{ -s+\oplam(t+B)} \in \Bap_{\cA}(G)$. Therefore, the autocorrelation $\gamma$ of $\delta_{ -s+\oplam(t+B)}$ exists with respect to $\cA$ \cite{LSS}.

We next show that $\gamma=\dens(\cL) \omega_{c(B)}$. This will imply in particular that $\gamma$ is independent of the choice of $(s,t)+\cL \in \cT$.

Note that for each $n$ we have
\begin{equation}\label{EQ3}
\delta_{ -s+\oplam(t+K_n)} \leq \delta_{ -s+\oplam(t+B)} \leq \delta_{ -s+\oplam(t+U_n)} \,.
\end{equation}

By Theorem~\ref{thm: min density model sets} the autocorrelation of $\delta_{ -s+\oplam(t+U_n)}$ exists with respect to $\cA$ and is equal to $\dens(\cL) \omega_{c(U_n)}$.
Same way, by \cite[Thm.~7]{BHS} the autocorrelation of $\delta_{ -s+\oplam(t+K_n)}$ exists with respect to $\cA$ and is equal to $\dens(\cL) \omega_{c(K_n)}$. Therefore, by \eqref{EQ3} we have
\begin{displaymath}
\dens(\cL) \omega_{c(K_n)} \leq \gamma \leq \dens(\cL) \omega_{c(U_n)} \,.
\end{displaymath}

Next, since $U$ is precompact, $\gamma$ is  supported inside the model set $\oplam(U-U)$, and hence inside $L= \pi_{G}(\cL)$.

Let $x \in L$. Then, for all $n$ we have

\begin{displaymath}
\dens(\cL) \theta_H(K_n \cap (x^\star+K_n)) \leq \gamma(\{ x \}) \leq  \dens(\cL) \theta_H(U_n \cap (x^\star+U_n)) \,.
\end{displaymath}

Then, for all $n$ we have
\begin{align*}
 & \left|\gamma(\{ x \})-\dens(\cL) \theta_H(B \cap (x^\star+B)) \right|   \leq \left|\gamma(\{ x\})-\dens(\cL) \theta_H(U_n \cap (x^\star+U_n)) \right| \\
  &+  \left|\dens(\cL)\theta_H(U_n \cap (x^\star+U_n))-\dens(\cL) \theta_H(B \cap (x^\star+B)) \right| \\
   & = \dens(\cL) \theta_H(U_n \cap (x^\star+U_n)) - \gamma(\{ x\}) + \theta_H(U_n \cap (x^\star+U_n))-\dens(\cL) \theta_H(B \cap (x^\star+B)) \\
   &\leq  2 \dens(\cL) \left( \theta_H(U_n \cap (x^\star+U_n))-\theta_H(K_n \cap (x^\star+K_n)) \right)  \\
   &=  2 \dens(\cL) \left( \int_{H} 1_{U_n}(t) 1_{U_n}(t-x^\star) \dd t -\int_{H} 1_{K_n}(t) 1_{K_n}(t-x^\star) \dd t \right) \\
    &=  2 \dens(\cL) \left( \int_{H} 1_{U_n}(t) 1_{U_n}(t-x^\star) \dd t -\int_{H} 1_{K_n}(t) 1_{U_n}(t-x^\star) \dd t \right)\\
    &+ 2 \dens(\cL) \left( \int_{H} 1_{K_n}(t) 1_{U_n}(t-x^\star) \dd t -\int_{H} 1_{K_n}(t) 1_{K_n}(t-x^\star) \dd t \right) \\
    &=  2 \dens(\cL) \left( \int_{H}\bigl( 1_{U_n}(t)- 1_{K_n}(t) \bigr) 1_{U_n}(t-x^\star) \dd t \right) \\
    &+ 2 \dens(\cL) \left( \int_{H} 1_{K_n}(t) \left(1_{U_n}(t-x^\star)- 1_{K_n}(t-x^\star) \right) \dd t \right) \,.\\
\end{align*}

Note here that since $K_n \subseteq U_n$ we have $1_{U_n}(t)- 1_{K_n}(t) \geq 0$ and $1_{U_n}(t-x^\star) \dd t 1_{K_n}(t-x^\star) \geq 0$ for all $t \in H$. Therefore,
\begin{align*}
  \left|\gamma(\{ x\})-\dens(\cL) \theta_H(B \cap (x^\star+B)) \right|  & \leq 2 \dens(\cL) \left( \int_{H} 1_{U_n}(t)- 1_{K_n}(t)\dd t \right)\\
  &+ 2 \dens(\cL) \left( \int_{H} 1_{U_n}(t-x^\star) - 1_{K_n}(t-x^\star)  \dd t \right) \\
  &=4 \dens(\cL) \left( \theta_H(U_n) - \theta_H(K_n)  \right)  < \frac{4 \dens(\cL) }{n} \,.\\
\end{align*}

As this holds for all $n \in N$ it follows that for all $x \in L$ we have
\[
\gamma(\{ x\})= \dens(\cL) \theta_H(B \cap (x^\star+B)) \,.
\]

This proves (c) .

\medskip

\textbf{(d)} Since $c(B)$ is continuous, positive definite and has compact support, we have $\widehat{c(B)} \in L^1(\widehat{H})$ \cite[Lemma~3.6]{CRS}.  Therefore, by \cite[Thm.~5.3]{CRS} we have
\begin{displaymath}
\widehat{\gamma}= \dens(\cL) \widehat{\omega_{c(B)}}=  \dens(\cL)^2 \omega_{\widecheck{c(B)}} \,.
\end{displaymath}

Now, since
\begin{displaymath}
\widecheck{c(B)} = \widecheck{ 1_{B} * \widetilde{1_{B}}}=\left| \widecheck{1_{B}} \right|^2
\end{displaymath}
the claim follows.

\medskip

\textbf{(e)} Let $(s,t) +\cL \in \cT$.

Since $\delta_{-s+\oplam(t+B)} \in \Bap_{\cA}(G)$ the Fourier--Bohr coefficients $a_{\chi}^\cA(\delta_{-s+\oplam(t+B)})$ exist by \cite[Thm.~3.36(b)]{LSS}.

By translation boundedness and the van Hove property we also have
\begin{displaymath}
a_{\chi}^\cA(\delta_{-s+\oplam(t+B)})= \chi(s) a_{\chi}^\cA(\delta_{\oplam(t+B)}) \,.
\end{displaymath}

Let us note first that for all $n$ we have
\begin{align*}
 &\left| a_{\chi}^\cA(\delta_{-s+\oplam(t+B)}) - a_{\chi}^\cA(\delta_{-s+\oplam(t+U_n)})\right|  \\
 &= \lim_k \frac{1}{|A_k|} \left| \int_{A_k} \overline{\chi(t)} \dd \delta_{-s+\oplam(t+B)}(t) - \int_{A_k} \overline{\chi(t)} \dd \delta_{-s+\oplam(t+U_n)}(t) \right|   \\
  & \leq \lim_k \frac{1}{|A_k|}\int_{A_k}  \left| \overline{\chi(t)}  \right| \dd \left|\delta_{-s+\oplam(t+B)} - \delta_{-s+\oplam(t+U_n)} \right|(t)    \\
     &= \lim_k \frac{1}{|A_k|} \left( \delta_{-s+\oplam(t+U_n)} - \delta_{-s+\oplam(t+B)}\right)(A_k)   \,. \\
\end{align*}

Now, since $(s,t)+\cL \in \cT$, the set $-s+\oplam(t+U_n)$ is a minimal density model set. Therefore, by Theorem\ref{thm: min density model sets} and (b) we have
\begin{equation}\label{eq4}
 \begin{split}
\nonumber \left| a_{\chi}^\cA(\delta_{-s+\oplam(t+B)}) - a_{\chi}^\cA(\delta_{-s+\oplam(t+U_n)})\right| &\leq  \dens(\cL) (\theta_H(t+U_n)- \theta_H(t+B)) \\
 &\leq  \dens(\cL) (\theta_H(U_n)- \theta_H(K_n))< \frac{\dens(\cL)}{n} \,.
 \end{split}
\end{equation}

Next, for all $\chi \notin \pi_{\hat{G}}(\cL^0)$, for all $n$ we have by \eqref{eq4}
\begin{align*}
  \left| a_{\chi}^\cA(\delta_{-s+\oplam(t+B)}) \right|  &=  \left| a_{\chi}^\cA(\delta_{-s+\oplam(t+B)}) - a_{\chi}^\cA(\delta_{-s+\oplam(t+U_n)})\right| \leq  \frac{\dens(\cL)}{n} \,.
\end{align*}

Therefore, for all $\chi \notin \pi_{\hat{G}}(\cL^0)$ we have $a_{\chi}^\cA(\delta_{-s+\oplam(t+B)})=0$.

Next, let  $\chi \in \pi_{\hat{G}}(\cL^0)$. Then
\begin{align*}
  &\left| a_{\chi}^\cA(\delta_{-s+\oplam(t+B)})-\dens(\cL)(\chi(s) \chi^\star(t)  ) \int_{B} \chi^\star(r) \dd r \right|    \leq  \left| a_{\chi}^\cA(\delta_{-s+\oplam(t+B)})- a_{\chi}^\cA(\delta_{-s+\oplam(t+U_n)}) \right|\\
  &+ \left| a_{\chi}^\cA(\delta_{-s+\oplam(t+U_n)})-\dens(\cL)(\chi(s)\chi^\star(t)  ) \int_{B} \chi^\star(r) \dd r \right|\\
   &\stackrel{\eqref{eq4}}{\leq}  \frac{\dens(\cL)}{n}+ \left| \chi(s) a_{\chi}^\cA(\delta_{\oplam(t+U_n)})-\dens(\cL)(\chi(s) \chi^\star(t)  ) \int_{B} \chi^\star(r) \dd r \right| \\
  &=  \frac{\dens(\cL)}{n}+ \left|  a_{\chi}^\cA(\delta_{\oplam(t+U_n)})-\dens(\cL)(\chi^\star(t)  ) \int_{B} \chi^\star(r) \dd r \right| \,. \\
\end{align*}

Now, since $\oplam(t+U_n)$ is minimal density model set, by Theorem~\ref{thm: min density model sets} we get
\begin{align*}
 & \left| a_{\chi}^\cA(\delta_{-s+\oplam(t+B)})-\dens(\cL)(\chi(s) \chi^\star(t)  ) \int_{B} \chi^\star(r) \dd r \right|  \\
 &  \leq   \frac{\dens(\cL)}{n}+ \left| \dens(\cL) \int_{t+U_n} \chi^\star(r) \dd r-\dens(\cL)(\chi^\star(t) ) \int_{B} \chi^\star(r) \dd r \right| \\
  &= \frac{\dens(\cL)}{n}+ \dens(\cL)\left|  \int_{U_n} \chi^\star(s+t) \dd s-(\chi^\star(t) ) \int_{B} \chi^\star(r) \dd r \right| \\
   &= \frac{\dens(\cL)}{n}+ \dens(\cL)\left|  \int_{H} \left(1_{U_n}(r)\chi^\star(r) - 1_{B} (r) \chi^\star(r)\right) \dd r \right| \\
      &= \frac{\dens(\cL)}{n}+  \dens(\cL)\int_{H} \left|  1_{U_n}(r) - 1_{B}(r)  \right|\dd r \\
         &= \frac{\dens(\cL)}{n}+ \dens(\cL)\left( \theta_{H}(U_n)-\theta_{H}(B)\right) \leq  \frac{2 \dens(\cL)}{n} \,. \\
\end{align*}

This proves the claim.

\medskip

\textbf{(f)} Follows from \cite[Thm.~3.36]{LSS}.

\medskip

\textbf{(g)} For each $(s,t) +\cL \in \cT$ the measure $\delta_{-s+\oplam(t+B)}$ is Besicovitch almost periodic. Then, by \cite[Thm.~6.13]{LSS} or \cite[Thm.~3.4 and Prop.~6.1]{LSS2} there exists an ergodic measure $m_{(s,t)+\cL}$ on $\XX(-s+\oplam(t+B)) \subseteq \extB$ such that $-s+\oplam(t+B)$ is generic for $m_{(s,t)+\cL}$.  We can consider $m_{(s,t)+\cL}$ as a measure on $\extB$.

We show now that $m_{(s,t)+\cL}$ does not depend on $(s,t)+\cL \in T$. Therefore, we can define $\m:= m_{(s,t)+\cL}$ for one, and hence all $(s,t)+\cL \in \cT$.

\smallskip

For each $n$ pick some $g_n \in \Cc(G)$ such that $1_{K_n} \leq g_n \leq 1_{U_n}$.

Define as usual the measure
\[
\omega_n:= \omega_{g_n}= \sum_{(x,x^\star) \in \cL} g_n(x^\star) \delta_x \,.
\]
Since $g_n \in \Cc(H)$, by Theorem~\ref{thm:omegah in sap} we have $\omega_n \in \SAP(G)$ and hence $\XX(\omega_n)$ is uniquely ergodic \cite{LR,LSS,LSS2,LS}. Let $\varpi_n$ be the unique ergodic measure on $\XX(\omega_n)$.

\smallskip

Note that all the measures $\omega_{g_n}$ are supported inside the Meyer set $\oplam(U)$. As $\|g_n\|_{\infty} =1$, it follows immediately that there exists a compact $K \subseteq G$ and some $C>0$ such that, all the hulls $\XX(\omega_{g_n})$ as well as $\extB$ are subsets of
\[
\cM_{C,K}:= \{ \mu : \| \mu \|_{K} \leq C \}
\]

This set is $G$-invariant and vaguely compact \cite{BL}. Let $\YY$ denote the set of real valued measures in $\cM_{C,K}$. Then, $\YY$ is $G$ invariant, vaguely compact and contains $\XX(g_n)$ and $\extB$. We can consider the measures $\varpi_n$ and $m_{(s,t)+\cL}$ as measures on $\YY$.

Now, as usual, for $\phi \in \Cc(G)$ we define $f_{\phi} : \YY \to \CC$ via
\[
f_{\phi}(\omega)=\omega*\phi(0) \,.
\]
Note that since $\YY$ only consists of real valued measures, we have
\[
\overline{f_{\phi}}= f_{\overline{\phi}} \,.
\]

Consider now, as in Proposition~\ref{prop:generic char}, the algebra $\AAA$ generated by $\{ f_{\phi}: \phi \in \Cc(G) \} \cup \{ 1_{\YY} \}$, which is dense in $C(\YY)$.

We show that for all $f \in \AAA$ and all $(s,t)+\cL \in \cT$ we have
\begin{equation}\label{eq5}
m_{(s,t)+\cL}(f)= \lim_n \varpi_n(f) \,.
\end{equation}
This will show that all $m_{(s,t)+\cL}$ agree on a dense subset of $C(\YY)$, and hence that they are all equal.

\smallskip

Since $m_{(s,t)+\cL}(1_{\YY})=1=\varpi_n(1_{\YY})$ we only need to show that \eqref{eq5} holds for functions $f$ of the form $f= f_{\phi_1}f_{\phi_2}\cdot...\cdot f_{\phi_n}$ for $n \in \NN$ and $\phi_1,.., \phi_n \in \Cc(G)$.

\smallskip

Next, note that $ (s,t)+ \cL \in \cT$ implies
\begin{align*}
 \delta_{-s+\oplam(K_n+t)}  &\leq  \delta_{-s+\oplam(B+t)} \leq   \delta_{-s+\oplam(U_n+t)} \\
 \delta_{-s+\oplam(K_n+t)}  &\leq  T_{-s}\omega_{T_tg_n} \leq   \delta_{-s+\oplam(U_n+t)} \\
\cM_{\cA}(  \delta_{-s+\oplam(U_n+t)}&- \delta_{-s+\oplam(K_n+t)})< \frac{\dens(\cL)}{n} \,.
\end{align*}

This immediately implies
\[
\uM_\cA(| T_{-s}\omega_{T_tg_n}- \delta_{-s+\oplam(B+t)}|) < \frac{\dens(\cL)}{n} \,.
\]

In particular,  for all $\phi \in \Cc(G)$ we have
\begin{equation}\label{eq: mean}
\bigl\| (T_{-s}\omega_{T_tg_n}) *\phi(x) - \delta_{-s+\oplam(B+t)}*\varphi(x) \bigr\|_{b,a,\cA}  < \frac{\dens(\cL)\|\varphi\|_1}{n} \,.
\end{equation}

Let $n \in \NN$ and $\phi_1,.., \phi_l \in \Cc(G)$. Then, similarly to \cite{BHS} we have
\begin{align*}
&\bigl\|  \prod_{k=1}^l f_{\phi_k} (T_x(T_{-s}\omega_{T_tg_n})) -\prod_{k=1}^l f_{\varphi_k} (T_x\delta_{-s+\oplam(B+t)}) \bigr\|_{b,1,\cA}  \\
&\leq \bigl\| \prod_{k=1}^l((T_{-s}\omega_{T_tg_n}))*\phi_k(x)-\prod_{k=1}^l \delta_{-s+\oplam(B+t)}*\phi_k(x)\bigr\|_{b,1,\cA} \\
&\leq \| \sum_{j=1}^l \left( \prod_{k=1}^{j-1} ((T_{-s}\omega_{T_tg_n}))*\phi_k(x)\right)\left(((T_{-s}\omega_{T_tg_n}))*\phi_j(x) \right.  \\
&\left.  - \delta_{-s+\oplam(B+t)}*\phi_j(x)\right) \left( \prod_{k=j+1}^{l}\delta_{-s+\oplam(B+t)}*\phi_k(x)\right) \|_{b,1,\cA} \\
&\leq \| \sum_{j=1}^l \left( \prod_{k=1}^{j-1} \|((T_{-s}\omega_{T_tg_n}))*\phi_k\|_{\infty} \right) \left|((T_{-s}\omega_{T_tg_n}))*\phi_j(x) \right. \\
&\left. - \delta_{-s+\oplam(B+t)}*\varphi_j(x)\right| \left( \prod_{k=j+1}^{l}\|\delta_{-s+\oplam(B+t)}*\varphi_k\|_\infty\right) \|_{b,1,\cA} \,.\\
\end{align*}

Now, since $0 \leq g_n \leq 1_{U_1}$ and $B \subseteq U_1$, the measures $(T_{-s}\omega_{T_tg_n})$ and $\delta_{-s+\oplam(B+t)}$ are equi translation bounded. Therefore, there exists a constant $C$, which only depends on $U_1$ and $\phi_1,.., \phi_l$ such that, for all $n,k$ we have
\begin{align*}
 \|((T_{-s}\omega_{T_tg_n}))*\phi_k\|_{\infty} &\leq C \,, \\
\|\delta_{-s+\oplam(B+t)}*\phi_k\|_\infty& \leq C \,.
\end{align*}

Therefore,
\begin{align}
&\bigl\| \prod_{k=1}^l f_{\phi_k} (T_x(T_{-s}\omega_{T_tg_n})) -\prod_{k=1}^l f_{\phi_k} (T_x\delta_{-s+\oplam(B+t)})\bigr\|_{b,1,\cA}   \nonumber\\
&\leq  C^{l-1} \sum_{j=1}^l \bigl\| ((T_{-s}\omega_{T_tg_n}))*\phi_j(x)- \delta_{-s+\oplam(B+t)}*\phi_j(x)\bigr\|_{b,1,\cA} \nonumber\\
&\leq  \frac{\dens(\cL)\left(\sum_{j=1}^l \|\phi_j\|_1\right) C^{l-1}}{n} \,. \label{eq 6}
\end{align}

Next, by Lemma~\ref{L1} we have
\begin{equation}\label{eq7}
\lim_k \frac{1}{|A_k|} \int_{A_k} \prod_{j=1}^l f_{\phi_j} T_x( (T_{-s}\omega_{T_tg_n})) \dd x= \varpi_n (\prod_{j=1}^l f_{\phi_j}) \,.
\end{equation}
Moreover, since $ (s,t)+ \cL \in \cT$ , the measure $\delta_{-s+\oplam(t+B)}$ is generic for $m_{(s,t)+\cL}$ and hence
\begin{equation}\label{eq8}
\lim_k \frac{1}{|A_k|} \int_{A_k} \prod_{j=1}^l f_{\phi_j} (T_x \delta_{-s+\oplam(t+B)})\dd x = m_{(s,t)+\cL} (\prod_{j=1}^l f_{\varphi_j}) \,.
\end{equation}

Combining \eqref{eq 6}, \eqref{eq7} and \eqref{eq8} we get
\[
\left| m_{(s,t)+\cL} (\prod_{j=1}^l f_{\phi_j}) -\varpi_n (\prod_{j=1}^l f_{\phi_j})  \right| \leq \frac{\dens(\cL)\left(\sum_{j=1}^l \|\phi_j\|_1\right) C^{l-1}}{n} \,.
\]

This shows \eqref{eq5}, and completes the proof of (g).

\medskip

\textbf{(h),(i), (j)} follow now from \cite{LSS,LSS2}.

\end{proof}

\begin{definition} Let $(G,H, \cL)$ be a CPS and $B \subseteq H$ be a precompact Borel set. We denote by $\m_B$ or simply $\m$ the measure given by Theorem~\ref{thm:main} (g).
\end{definition}

\section{Properties of $\m$}\label{Sect: prop m}

Consider a CPS $(G,H,\cL)$ and a precompact Borel set $B \subseteq H$.

Let $\m$ be the ergodic measure on $\extB$ from Theorem~\ref{thm:main} (g) and define
\begin{align*}
\YY_{\m}&:= \{ \Gamma \in \extB : \Gamma \mbox{ is generic for } \m \} \\
\YY_{b,\m}&= \YY_\m \cap \Bap_{\cA,ps} (G) \,.
\end{align*}

The following result follows trivially from translation boundedness.

\begin{lemma}
\begin{itemize}
  \item[(a)] Let $\Gamma \in \YY_{\m}$ and $t \in G$. Then $T_t \Gamma \in \YY_{\m}$.
  \item[(b)] Let $\Gamma \in \YY_{b,\m}$ and $t \in G$. Then $T_t \Gamma \in \YY_{b,\m}$.
\end{itemize}
\end{lemma}
\qed

Note that in general $\YY_\m$ is not closed in $\extB$, but it is closed for regular model sets.

\medskip

We can now list the following properties of $\m$ and the sets $\YY_{\m}$ and $\YY_{b,\m}$.

\begin{theorem}\label{thm:main 2}
\begin{itemize}
  \item[(a)] $m(\YY_\m)= m(\YY_{b,\m}) =1$.
  \item[(b)] For each $\omega \in \YY_\m$ the autocorrelation of $\omega$ exists with respect to $\cA$ and is equal to $\gamma$.
  \item[(c)] For each $\chi \in \widehat{G}$ with $\widehat{\gamma}(\{ \chi \}) \neq 0$ the Fourier--Bohr coefficients
  \[
  a_\chi^\cA : \YY_{b,\m} \to \CC
  \]
  defines a measurable eigenfunction for $\chi$. Moreover, for all $\omega \in \YY_{b,\m}$ we have
  \[
  \widehat{\gamma}(\{\chi \})= \left| a_\chi^{\cA}(\omega) \right|^2 \,.
  \]
  \item[(d)] $\YY_{\m}$ is pseudo-minimal in the following sense: if $\Gamma_1, \Gamma_2 \in \YY_{\m}$ then
  \[
  \Gamma_1 \in \XX(\Gamma_2) \, \mbox{ and } \, \Gamma_{2} \in \XX(\Gamma_1) \,.
  \]
  \item[(e)] There exists a set $\cT \subseteq \TT$ of full measure such that, for all $(s,t)+\cL \in \cT$ we have
  \[
  -s+\oplam(t+B) \in \YY_{b,m} \,.
  \]
  \item[(f)] If $\Gamma \in \YY_\m$ then $\dens_{\cA}(\Gamma)= \dens(\cL) \theta_H(B) \,.$
\end{itemize}
\end{theorem}
\begin{proof}
\textbf{(a)} By \cite[Theorem.~6.10]{LSS} we have $m(\extB \cap \Bap_{\cA}(G))=1$.

Next, since $G$ is second countable, there exists a countable dense set $Q \subseteq C(\extB)$. By Birkhoff ergodic Theorem, for each $f \in Q$, there exists a set $X_{f} \subseteq \YY$ of full measure, such that, for all $\omega \in X_{f}$ we have
\[
\int_{\YY} f(\nu) \dd \m(\nu) = \lim_k \frac{1}{|A_k|} \int_{A_k} f(T_t\omega) \dd t  \,.
\]

Then, the set
\[
X:= \bigcup_{f \in Q} X_f
\]
has full measure in $\extB$ and , by Prop~\ref{prop:generic char2}, each $\omega \in X$ is generic for $\m$.

Therefore, $X \subseteq \YY_{\m}$. Since $\m(X)=\m(\extB \cap \Bap_{\cA}(G))=1$, the claim follows.

\medskip

\textbf{(b)} By \cite{BL}, for all $\psi, \varphi \in \Cc(G)$ and all $t \in G$ we have
\[
\gamma*\varphi*\tilde{\psi}(t)= \langle f_\varphi , T_t f_\psi \rangle = \int_{\extB} f_\varphi(\omega) \overline{T_t f_\psi (\omega)} \dd \m(\omega) \,.
\]

Next, let $\mu \in \YY_\m$. Then,
\begin{align*}
 \gamma*\varphi*\tilde{\psi}(t)&= \int_{\extB} f_\varphi(\omega) \overline{T_t f_\psi (\omega)} \dd m(\omega) \\
  &= \lim_m \frac{1}{|A_k|} \int_{A_k} f_\varphi(T_s \mu) \overline{T_t f_\psi (T_s \mu)} \dd s\\
  &= \lim_m \frac{1}{|A_k|} \int_{A_k} \mu*\varphi(s) \overline{  \mu *\psi (s-t)}  \dd s\\
&= \lim_m \frac{1}{|A_k|} \int_{A_k} \mu*\varphi(s) \widetilde{  \mu *\psi (t-s)}  \dd s \,.\\
\end{align*}

Then, by \cite{LSS2} the autocorrelation of $\mu$ exists with respect to $\cA$ and is equal to $\gamma$.

\textbf{(c)} Follows from \cite[Theorem.~6.10]{LSS}.

\medskip

\textbf{(d)}We have $\m(\XX(\Gamma_{2})) \in
\{0,1\}$. Since $\Gamma_2$ is generic for $\m$ we must have $\m(\XX(\Gamma_{2}))\neq 0$ and hence $\m(\XX(\Gamma_{2}))=1$.

This implies that $\supp(\m) \subseteq \XX(\Gamma_2)$.

Then, using that $\Gamma_1$ is generic for $\m$ we have
\[
\Gamma_1 \in \supp(\m) \subseteq \XX(\Gamma_2) \,.
\]

Switching the roles of $\Gamma_1, \Gamma_2$ gives the other implication.

\medskip

\textbf{(e)} Follows from Theorem~\ref{thm:main}.

\textbf{(f)} We have
\[
\dens(\cL) \theta_H(B)= \gamma(\{ 0\})\,.
\]
Now, since $\gamma$ is the autocorrelation of $\Gamma$ we have
\[
\gamma(\{0\})= \dens_{\cA}(\Gamma) \,.
\]

\end{proof}

\section{The set $\YY_{app}$ and the relation to $\YY_{b,m}$}

Our next goal is to try to define a factor mapping $\pi : (\extB, G, \m)$ to a torus. We could try to define $\pi$ in a natural way from the set
\[
Z:= \{ -s+\oplam(t+B) : (s,t)+\cL \in \cT \} \,,
\]
but this is problematic. Indeed, while $\cT$ has full measure in $\TT$, we see no reason why $Z$ would have full measure in $(\extB, G, \m)$. Moreover, elements in $Z$ can come from different $(s,t)+\cL \in \cT$, which means that the mapping $\pi$ could only be defined on a factor of $\TT$.

Finally, the situation is much worse: many model sets with compact window are known to be hereditary. Then, given an arbitrary $\Gamma \in Z$, all the subsets $\Lambda \subseteq \Gamma$ with $\dens(\Lambda)= \dens(\Gamma)$ become generic for $\m$, and should also be included somehow in $Z$. This is exactly what we do in this section.

In this section, we define a new set $\YY_{app}$. Heuristically, this set consists of all $\Lambda$ for which there exists some $(s,t)+\cL$ in $\TT$, such that
\begin{itemize}
  \item{} $-s+\oplam(t+B)$ is generic for $\m$.
  \item{} $\dens(\Lambda \Delta (-s+\oplam(t+B)))=0$.
\end{itemize}

To identify these points, we rely on the proof of Theorem~\ref{thm:main}: in the notation of that Theorem, for any such $(s,t)$, the measure $\delta_{\Gamma}$ (and hence $\delta_\Lambda$) must be the limit in the mean of $T_{s} \omega_{g_n}$. This is exactly how we are going to define $\YY_{app}$, as being the set of pointsets which can be approximated this way.

With this definition we will show that $\YY_{app} \subseteq \YY_{b,\m}$ and that $\YY_{app}$ does indeed contain all elements of the form $\{ -s+\oplam(t+B) : (s,t)+\cL \in \cT\}$. Next, we define in the natural way a mapping from $\YY_{app}$ to a factor of $\TT$. In the case of compact windows, we then show that $\YY_{app}=\YY_{b,\m}$ consist exactly of the elements in $\extB$ of maximal density, and that the mapping $\pi$ is continuous. Since $\m(\YY_{b,\m})=1$, we get that this is a Borel factor map, with very nice properties.

\bigskip

Let us proceed to the construction and study of $\YY_{app}$.

As usual, we fix a  $(G,H,\cL)$ and a precompact Borel set $B \subseteq H$, $\extB$  denotes the extended hull
and $\m$ is the ergodic measure from Theorem~\ref{thm:main} (g).

As in the proof of Theorem~\ref{thm:main}, pick compact set $K_n$ and open precompact set $U_n \subseteq U$ such that $K_n \subseteq B \subseteq U_n$ and
\[
\theta_{H} (U_n \backslash K_n) < \frac{1}{n} \,.
\]
Let $g_n \in \Cc(H)$ be so that $1_{K_n} \leq g_n \leq 1_{U_n}$.

\smallskip

We start by giving an intrinsic characterisation of $g_n$, which will allow us show below that all our definition are independent of the choice of $K_n,U_n$ and $g_n$, respectively.

\begin{lemma} \label{l3}
\begin{itemize}
  \item[(a)]$g_n$ converges to $1_{B}$ in $L^1(H)$.
  \item[(b)] If $h_n \in \Cc(H)$ converges to $1_B$ in $L^1(H)$ then, for all $(s,t) \in G \times H$ we have
  \[
 \lim_n \uM_{\cA}(\left| T_{-s}\omega_{T_t g_n} - T_{-s}\omega_{T_t h_n}\right|) = 0 \,.
  \]
\end{itemize}
\end{lemma}
\begin{proof}
\textbf{(a)}
We have $1_{U_n} \leq g_n \leq 1_{K_n}$. Therefore
\begin{align*}
g_n - 1_B & \leq 1_{U_n}-1_{B} \\
1_{B}-g_n &\leq 1_{B}-1_{K_n} \,.
\end{align*}

Since $ 1_{U_n}-1_{B} \geq 0$ and $1_{B}-1_{K_n} \geq 0$, we immediately get
\[
\left| g_n -1_{B} \right| \leq \max\{  1_{U_n}-1_{B}, 1_{B}-1_{K_n} \} \leq \left(1_{U_n}-1_{B}\right) + \left(1_{B}-1_{K_n} \right) =1_{U_n} -1_{K_n}\,.
\]

We thus have
\[
\| g_n -1_{B} \|_1 \leq \theta_H(U_n)-\theta_H(K_n) < \frac{1}{n} \,.
\]

\textbf{(b)} Let $(s,t) \in G \times H$ be arbitrary. Then
\begin{align*}
\left|T_{-s}\omega_{T_t g_n} - T_{-s}\omega_{T_t h_n}\right| &= \left| \left(T_{-s} \sum_{x \in L} (T_tg_n(x^\star) \delta_x) \right)-\left(T_{-s} \sum_{x \in L} (T_th_n(x^\star) \delta_x) \right) \right| \\
&= \left| T_{-s} \left( \sum_{x \in L} \bigl(T_tg_n(x^\star) -T_th_n(x^\star)\bigr) \delta_x \right) \right| \\
&=   T_{-s} \left|\sum_{x \in L} \bigl(T_tg_n(x^\star) -T_th_n(x^\star)\bigr) \delta_x  \right| \\
&=   T_{-s} \sum_{x \in L}\left| T_tg_n(x^\star) -T_th_n(x^\star) \right| \delta_x  \\
&= T_{-s}\omega_{|T_tg_n-T_th_n|} \,.
\end{align*}

Since $|T_tg_n-T_th_n| \in \Cc(H)$ we get \cite{CRS}
\begin{align*}
&\uM_{\cA}(\left| T_{-s}\omega_{T_t g_n} - T_{-s}\omega_{T_t h_n}\right|)= \dens(\cL) \int_{H} |T_tg_n-T_th_n| (y) \dd y \\
&= \dens(\cL) \int_{H} |g_n(y)-h_n(y)| \dd y \\
&= \dens(\cL) \| g_n -h_n \|_1 \leq \dens(\cL) \left( \|g_n -1_B \|_1+\|h_n -1_{B} \|_1 \right) \,.
\end{align*}
\end{proof}

\medskip

As an immediate consequence we get.

\begin{corollary}\label{C1} Let $h_n \in \Cc(H)$ be so that $\|h_n -1_{B} \|_1 \to 0$, let $\Gamma \in \extB$ and $(s,t) \in G \times H$.  Then
\begin{displaymath}
\lim_n \uM_{\cA}(\left| T_{-s}\omega_{T_t g_n} - \delta_{\Gamma}\right|) = 0 \,.
\end{displaymath}
if and only if
\begin{displaymath}
\lim_n \uM_{\cA}(\left| T_{-s}\omega_{T_t h_n} - \delta_{\Gamma}\right|) = 0 \,.
\end{displaymath}
\end{corollary}

\smallskip

We can now define $\YY_{app}$.

\begin{definition}
Define $\YY_{app}$ to be the set of all $\Gamma \in \extB$ for which, there exists some $(s,t) +\cL \in \TT$ such that
\begin{equation}\label{eq9}
\lim_n \uM_{\cA}(\left| T_{-s}\omega_{T_t g_n} - \delta_{\Gamma}\right|) = 0 \,.
\end{equation}
\end{definition}

\begin{remark} Corollary~\ref{C1} says that $\YY_{app}$ does not depend on the choice of $g_n$. Moreover, for each $\Gamma \in \YY_{app}$ the choice of $(s,t) \in G \times H$ such that \eqref{eq9} holds does not depend on $g_n$.

Moreover, if instead of $g_n$ we would use any $h_n \in \Cc(H)$ such that $h_n \to 1_{B}$ in $L^1(H)$, we would get exactly the same definition for $\YY_{app}$ and exactly the same choices of $(s,t)$.
\end{remark}

\medskip

Since $\uM_{\cA}$ is translation invariant for translation bounded measures, we immediately get:
\begin{lemma}\label{l4} Let $\Gamma \in \YY_{app}$ and $t \in G$. Then $T_t \Gamma \in \YY_{app}$. \qed
\end{lemma}

Now, we can prove the following result:

\begin{proposition}
\begin{itemize}
  \item [(a)] $\YY_{app} \subseteq \YY_{b,\m}$.
  \item [(b)] There exists a set $\cT \subseteq \TT$ of full measure such that, for all $(s,t)+\cL \in \cT$ we have $-s+\oplam(t+B) \in \YY_{app}$.
\end{itemize}
\end{proposition}
\begin{proof}
\textbf{(a)}

Exactly as in the proof of Theorem~\ref{thm:main}, set $\omega_n := \omega_{g_n}$ and $\varpi_n$ the unique ergodic measure on $\XX(\omega_n)$.

Note here that, by Lemma~\ref{L1}, for all $(s,t)+\cL \in \TT$ we have $T_{-s}\omega_{T_t g_n} \in \XX(\omega_n)$.
Exactly as in the Proof of Theorem~\ref{thm:main} we can find some $C,K$ such that $\XX(\omega_n)$ and $\extB$ are subsets of
\[
Z:= \cM_{C,K}(G) \,.
\]

Let $\AAA$ the algebra generated by $\{ F_\varphi : \varphi \in \Cc(G) \} \cup \{1_{Z}\}$.

Pick some $\Gamma \in \YY_{app}$.

Let $\eps >0$ and let $f \in \AAA$. Exactly as in the proof of Theorem~\ref{thm:main} (g), \eqref{eq9} implies that there exists some
$N_1$ such that, for all $n>N_1$ we have
\begin{equation}\label{eq10}
\bigl\| f( T_{r-s}\omega_{T_t g_n}) -  f( T_{r}\delta_{\Gamma}) \bigr\|_{b,1,\cA} < \frac{\eps}{6} \,.
\end{equation}
Next, by the proof of Theorem~\ref{thm:main} (g) we have
\[
\lim_n \varpi_n(f)=\m(f) \,.
\]
Therefore, there exists some $N_2$ such that, for all $n> N_2$ we have
\[
\left| \varpi_n(f)-\m(f) \right| < \frac{\eps}{3} \,.
\]

Fix some $n > \max\{N_1, N_2 \}$.

Since $(\XX(\omega_n, \varpi_n)$ is uniquely ergodic,
\[
\lim_k \frac{1}{|A_k|}\int_{A_k} f( T_{r-s}\omega_{T_t g_n}) \dd r = \varpi_n(f) \,.
\]

Therefore, there exists some $M_1$ such that, for all $k >M_1$ we have
\[
\left| \frac{1}{|A_k|}\int_{A_k} f( T_{r-s}\omega_{T_t g_n}) \dd r - \varpi_n(f) \right|<\frac{\eps}{3} \,.
\]
By \eqref{eq10} there exists some $M_2$ such that for all $k >M_2$ we have
\[
\frac{1}{|A_k|} \left| \int_{A_k} f( T_{r-s}\omega_{T_t g_n}) \dd r -  \int_{A_k}f( T_{r}\delta_{\Gamma}) \dd r \right| < \frac{\eps}{3} \,.
\]

Let $M =\max\{M_1,M_2\}$

Then, for all $k >M$ we have
\begin{align*}
\left| \frac{1}{|A_k|} \ f( T_{r}\delta_{\Gamma}) \dd r -\m(f) \right| & \leq \frac{1}{|A_k|} \left| \int_{A_k} f( T_{r-s}\omega_{T_t g_n}) \dd r -\int_{A_k} f( T_{r}\delta_{\Gamma}) \dd r \right| \\
& +\left| \frac{1}{|A_k|}\int_{A_k} f( T_{r-s}\omega_{T_t g_n}) \dd r - \varpi_n(f) \right|+\left| \varpi_n(f)-\m(f) \right|  \\
&< \frac{\eps}{3}+\frac{\eps}{3}+\frac{\eps}{3}= \eps \,.
\end{align*}

Therefore, we showed that for each $\eps >0$ there exists some $M$ such that, for all $k>M$ we have
\[
\left| \frac{1}{|A_k|} \ f( T_{r}\delta_{\Gamma}) \dd r -\m(f) \right| < \eps \,.
\]

This shows that for all $f \in \AAA$ we have
\[
\lim_k  \frac{1}{|A_k|} \ f( T_{r}\delta_{\Gamma}) \dd r = \m(f) \,.
\]

Then, by Prop~\ref{prop:generic char} we have $\delta_{\Gamma} \in \YY_{\m}$.

Finally, for all $\varphi \in \Cc(G)$ we have
\[
\|  T_{-s}\omega_{T_t g_n} *\varphi - \delta_{\Gamma} *\varphi \|_{b,1,\cA} \leq \uM_{\cA}(| T_{-s}\omega_{T_t g_n} - \delta_{\Gamma}|) \| \varphi \|_1  \to 0\,.
\]

Since $ T_{-s}\omega_{T_t g_n} *\varphi \in SAP(G) \subseteq Bap_{\cA}(G)$ and $(Bap_{\cA}(G), \| \, \|_{b,1, \cA})$ is a Banach space \cite{LSS}, we get that $\delta_{\Gamma} \in \Bap_{\cA}(G)$. This proves (a).

\medskip

\textbf{(b)} this is the proof of Theorem~\ref{thm:main} (g).
\end{proof}

\section{The reduced torus $\TT_{red}$ and the mapping $\pi : \YY_{app} \to \TT_{red}$.}

Let $(G,H, \cL)$ be a CPS, and $B \subseteq H$ be a precompact Borel set. We will denote by
\[
H_B:= \{ t \in H : \theta_H( B \Delta (t+B))=0 \} \,.
\]

We will also denote by $H_{app}$ the set of all $t \in H$ with the property that
\begin{equation}\label{eq11}
\lim_n \uM_{\cA}(\left| \omega_{g_n} -  \omega_{T_t g_n} \right|) = 0 \,.
\end{equation}

We start by proving the following lemma.

\begin{lemma}\label{lemma H_2} Let $c_B$ denote the covariogram of $B$. Then, the following are equivalent for $t \in H$.
\begin{itemize}
  \item[(i)] $t \in H_B$.
  \item[(ii)] $T_t1_B=1_B$ in $L^1(H)$.
  \item[(iii)] $c_B(t)=c_B(0)$.
  \item[(iv)] $t$ is a period for $c_B$.
  \item[(v)] $t \in H_{app}$.
\end{itemize}
In particular, $H_W=H_{app}$ is a closed subgroup of $H$.
\end{lemma}
\begin{proof}
\textbf{(i)} $\Longleftrightarrow$ \textbf{(ii)} and \textbf{(iv)} $\Longrightarrow$ \textbf{(iii)} are obvious.

\textbf{(iii)} $\Longrightarrow$ \textbf{(iv)} follows from Krein's inequality \cite{BF}, as $c_B=1_B*\widetilde{1_{B}}$ is positive definite.

\textbf{(i)} $\Longrightarrow$ \textbf{(iii)}

Note that for all $t \in H_B$ we have
\[
c_B(t)= \theta_H(B \cap (t+B)) = \theta_H(B) - \theta_H(B \backslash (t+B)) \,.
\]
Now, since $\theta_H( B \Delta (t+B))=0$ we have $\theta_H(B \backslash (t+B))=0$ and hence
\[
c_B(t)= \theta_H(B)=c_B(0) \,.
\]

\textbf{(iii)} $\Longrightarrow$ \textbf{(i)}

\begin{displaymath}
c_B(t) = c_B(0) \Rightarrow \theta_H(B \cap (t+B)) =\theta_H(B)=\theta_H(t+B) \,. \\
\end{displaymath}

This implies that $\theta_H(B \backslash B \cap (t+B)) =\theta_H((t+B) \backslash B \cap (t+B))=0$ and hence $ \theta_H( B \Delta (t+B))=0$.

\textbf{(ii)} $\Longleftrightarrow$ \textbf{(v)}
Let $t \in H$.

Note that
\begin{align*}
\left| \omega_{g_n} -  \omega_{T_t g_n} \right| &= \left| \sum_{ x \in L} g_n(x^\star)\delta_x -\sum_{ x \in L} g_n(x^\star-t)\delta_x \right| \\
&=   \sum_{ x \in L} \left| g_n(x^\star) - g_n(x^\star-t) \right|\delta_x = \omega_{|g_n -T_tg_n|} \,.
\end{align*}

Therefore,
\[
\uM_{\cA}(\left| \omega_{g_n} -  \omega_{T_t g_n} \right|) = \dens(\cL) \int_{H} \left| g_n(y) - T_t g_n(y) \right| \dd y \,.
\]

Since $\left|g_n -T_t g_n \right| \leq 1_{U_1}+T_t 1_{U_1} \in L^1(H)$ and $ \left| g_n(y) - T_t g_n(y) \right|$ converges in $L^1(H)$ to $|1_B -T_t1_B|$, by the Dominated convergence theorem we get

\begin{equation}\label{eq12}
\lim_n \uM_{\cA}(\left| \omega_{g_n} -  \omega_{T_t g_n} \right|) = \dens(\cL) \| 1_B-T_t1_B \|_1 \,.
\end{equation}
The claim follows.

This proves the equivalence and that $H_B$ is a subgroup of $H$.

Finally, since $c_B \in \Cc(H)$, $H_B$ is closed in $H$.

\end{proof}

\begin{lemma} If $B$ is precompact and $\theta(B) >0$ then $\cL + \{ 0 \} \times H_B$ is closed in $G \times H$.
\end{lemma}
\begin{proof}

Since $\theta(B) >0$ we immediately get that $H_B \subseteq B-B$ and hence $H_B$ is compact. Then, $\cL$ is closed and $\{ 0\} \times H_B$ is compact, and the claim follows
\end{proof}

\begin{remark} If $\theta(B)=0$ then $H_B=H$ and $\cL + \{ 0 \} \times H_B= L \times H$. This is typically not closed in $G \times H$.
\end{remark}

We can now give the following definition
\begin{definition} Let $(G, H, \cL)$ be a CPS and $B$ a pre-compact window with $\theta_H(B)>0$. Define the \textbf{reduced torus} as
\[
\TT_{red}:= (G \times H) / ( \cL + \{ 0 \} \times H_B) =\TT / (\{ 0\} \times H_B) \,.
\]
\end{definition}

The following Lemma is straightforward.

\begin{lemma} Let $B$ be any Borel set with $\theta_H(B)>0$, let $\Gamma \in \YY_{app}$ and let $(s,t) \in G\times H$ be so that \eqref{eq9} holds. Let $(s',t') \in G \times H$. Then, \eqref{eq9} holds for $(s',t')$ if and only if
\[
(s,t)+\cL+\{ 0\} \times H_{B}=(s',t')+\cL+\{ 0\} \times H_{B} \,.
\]
\end{lemma}
\begin{proof}
$\Longrightarrow$:
Since
\begin{align*}
\lim_n \uM_{\cA}(\left| T_{-s}\omega_{T_t g_n} - \delta_{\Gamma}\right|) &= 0  \,,\\
\lim_n \uM_{\cA}(\left| T_{-s'}\omega_{T_{t'} g_n} - \delta_{\Gamma}\right|) &= 0 \,,\\
\end{align*}
by the triangle inequality we have
\begin{align*}
\lim_n \uM_{\cA}(\left| T_{-s}\omega_{T_t g_n} - T_{-s'}\omega_{T_{t'} g_n} \right|) &= 0 \,.
\end{align*}

We claim that $s-s' \in L$.

Assume by contradiction that  $s-s' \notin L$. Then, $T_{-s}\omega_{T_t g_n}$ is supported on $-s+L$ and  $T_{-s'}\omega_{T_{t'} g_n}$ is supported on $-s'+L$. Since $L$ is a group, and $s-s' \notin L$, it follows that  $T_{-s}\omega_{T_t g_n}$ and  $T_{-s'}\omega_{T_{t'} g_n}$ are pure point measures with disjoint support. Therefore
 \[
 \left| T_{-s}\omega_{T_t g_n} - T_{-s'}\omega_{T_{t'} g_n} \right|= \left| T_{-s}\omega_{T_t g_n} \right|+\left|  T_{-s'}\omega_{T_{t'} g_n} \right| \,.
 \]
Using the fact that $g_n \geq 0$ we get
 \begin{align*}
\uM_{\cA}(\left| T_{-s}\omega_{T_t g_n} - T_{-s'}\omega_{T_{t'} g_n} \right|) =2 \int_H g_n(h) \dd h \,.
\end{align*}

Therefore,
\[
0 =\lim_n \uM_{\cA}(\left| T_{-s}\omega_{T_t g_n} - T_{-s'}\omega_{T_{t'} g_n} \right|)= \lim_n 2 \int_H g_n(h) \dd h = 2 \theta(B)
\]
which contradicts $\theta_H(B) >0$. Therefore, we got a contradiction and hence $s-s' \in L$.

\smallskip

It follows that
\begin{align}
0&=\lim_n \uM_{\cA}(\left| T_{-s}\omega_{T_t g_n} - T_{-s'}\omega_{T_{t'} g_n} \right|) \nonumber \\
&=\lim_n \uM_{\cA}(\left| \omega_{T_t g_n} - T_{s-s'}\omega_{T_{t'} g_n} \right|) \nonumber \\
&\stackrel{\eqref{eq:translate omegah}}{=\joinrel=\joinrel=}\lim_n \uM_{\cA}(\left| \omega_{T_t g_n} -\omega_{T_{(s-s')^\star+t'} g_n} \right|)  \,. \label{eq21}
\end{align}

Next, exactly as in the proof of Lemma~\ref{l3} and Lemma~\ref{lemma H_2} we have
\begin{align}
&\uM_{\cA}(\left| \omega_{T_t g_n} -\omega_{T_{(s-s')^\star+t'} g_n} \right|) = \uM_{\cA}( \omega_{|T_t g_n-T_{(s-s')^\star+t'} g_n|}) \label{eq22}\\
&= \dens(\cL) \int_H|T_t g_n-T_{(s-s')^\star+t'} g_n|(h) \dd h = \dens(\cL) \int_H|g_n-T_{(s-s')^\star+t'-t} g_n|(h) \dd h \nonumber\\
&=  \uM_{\cA}(\left| \omega_{g_n} -\omega_{T_{(s-s')^\star+t'-t} g_n} \right|) \,.  \nonumber
\end{align}

This shows that $(s-s')^\star+t'-t \in H_W$ and hence
\[
(s-s',t-t')=(s-s', (s-s')^\star)-(0, (s-s')^\star+t'-t ) \in \cL+ \{ 0 \} \times H_W \,.
\]

This completes the proof of the first implication.

\bigskip

$\Longleftarrow$:

We are given that
\begin{align*}
\lim_n \uM_{\cA}(\left| T_{-s}\omega_{T_t g_n} - \delta_{\Gamma}\right|) &= 0 \\
(s-s',t-t') &\in \cL+ \{ 0 \} \times H_B \,.
\end{align*}
It follows from here that
\[
(0, (s-s')^\star+t'-t ) \in \cL+ \{ 0 \} \times H_B \,.
\]
Since the projection of $\cL$ on $G$ is one to one, this implies that $(s-s')^\star+t'-t  \in  \times H_B=H_{app}$ and hence
\[
\uM_{\cA}(\left| \omega_{g_n} -\omega_{T_{(s-s')^\star+t'-t} g_n} \right|)=0 \,.
\]

Repeating the computations from \eqref{eq21} and \eqref{eq22} in reverse order we get
\begin{displaymath}
0=\lim_n \uM_{\cA}(\left| T_{-s}\omega_{T_t g_n} - T_{-s'}\omega_{T_{t'} g_n} \right|) \\
\end{displaymath}
Combining this relation with
\begin{align*}
\lim_n \uM_{\cA}(\left| T_{-s}\omega_{T_t g_n} - \delta_{\Gamma}\right|) &= 0 \\
\end{align*}
we get
\begin{align*}
\lim_n \uM_{\cA}(\left| T_{-s'}\omega_{T_{t'} g_n} - \delta_{\Gamma}\right|) &= 0 \,, \\
\end{align*}
as claimed.
\end{proof}

This allows us to define the following mapping
\begin{definition} Let $(G,H,\cL)$ be a CPS, $B \subseteq H$ a precompact Borel set with $\theta_H(B)>0$. Let $\Gamma \in \YY_{app}$ and let $(s,t) \in G \times H$ be so that \eqref{eq9} holds. Define
\[
\pi(\Gamma):=(s,t)+\cL+ \{ 0 \} \times H_{B}
\]
This gives a mapping
\[
\pi : \YY_{app} \to \TT_{red} \,.
\]
\end{definition}

The following Lemma is obvious.

\begin{lemma}Let $\Gamma \in \YY_{app}$ and $s \in G$. Then
\[
\pi(s+\Gamma)=(s,0)+\pi(\Gamma) \,,
\]
that is, $\pi$ defines a $G$-invariant mapping
\[
\pi : \YY_{app} \to \TT_{red} \,.
\]
\end{lemma}

For model sets with compact windows, we will give in the next section an equivalent definition for this mapping, which is much easier to understand and use. We will the new definition, together with properties of $\YY_{app}$ to show that for compact windows, $\pi$ is continuous and defines
a Borel factor from $(\extB,m)$ to $\TT_{red}$ which induces an isometric isomorphism between $L^2(\extB,m)$ and $L^2(\TT_{red})$.

\section{Weak model sets}

We start with the following preliminary lemma, which has been proven in various places (see for example \cite{JBA,Martin2}).
As we need the result for the extended hull, and the result has been proven before just for the usual hull, we include the proof for completeness.
We should emphasize here that working with the extended hull makes no difference for the proof.

\begin{lemma}\label{lem torus} Let $(G, H, \cL)$ be a CPS and $W \subseteq H$ be a compact window and let $\Gamma \in \extW$. Then, there exists some $(s,t)+\cL \in \TT$ such that
\[
\Gamma \subseteq -s+\oplam(t+W) \,.
\]
\end{lemma}
\begin{proof}

The proof is similar to \cite[Lemma.~4.1]{Martin2} or \cite[Thm.~5.9]{JBA}.

First, we can fix some neighbourhood of zero $V$ such that, for all $(s,t)+\cL$ the set $-s+\oplam(t+W)$ is $V$-uniformly discrete.

Since $\extB$ is closed under translations, we can assume without loss of generality that $0 \in \Gamma$.

Let now $(s_n, t_n) +\cL \in \TT$ be so that $-s_n + \oplam(t_n+W) \to \Gamma$. We know that such elements exist. By compactness, the sequence $(s_n,t_n)+\cL$ has a convergent subnet $(s_\alpha, t_\alpha)+\cL$ which converges to some $(s,t)+\cL \in \TT$.

We show that
\[
\Gamma \subseteq -s+\oplam(t+W) \,.
\]

Let $x \in \Gamma$ be arbitrary.

Since $x+V$ is an open neighbourhood of $x$, by the definition of the rubber topology there exists some $\alpha_0$ such that for all $\alpha > \alpha_0$ the set $(x+V) \cap   (-s_\alpha + \oplam(t_\alpha+W))$ is non-empty. Therefore, there exists some $x_\alpha \in (x+V) \cap  ( -s_\alpha + \oplam(t_\alpha+W))$, which must be unique by the choice of $V$.
This means
\[
 (x+V) \cap   (-s_\alpha + \oplam(t_\alpha+W))= \{ x_ \alpha \} \, \forall \alpha > \alpha_0 \,.
\]

We show that $\{ x_{\alpha} \}_{\alpha > \alpha_0}$ converges as a net to $x$.

Let $U$ be any neighbourhood of $x$ such that $x \in U \subseteq x+V$. Then, since $-s_\alpha+\oplam(t_\alpha+W)$ converges to $\Gamma$, and since $x \in \Gamma$, there exist some
$\alpha_U> \alpha_0$ such that for all $\alpha >\alpha_U$ the set $U \cap  ( -s_\alpha + \oplam(t_\alpha+W))$ is non-empty. Therefore, as
\[
U \cap   -s_\alpha + \oplam(t_\alpha+W) \subseteq x+V \cap   -s_\alpha + \oplam(t_\alpha+W) = \{ x_ \alpha \}
\]
it follows that $U \cap   -s_\alpha + \oplam(t_\alpha+W)  = \{ x_ \alpha \} $. This shows that for all $\alpha > \alpha_U$ we have $x_\alpha \in U$. This proves that
\[
\lim_{\alpha > \alpha_0} x_\alpha =x \,.
\]

Next, for all $\alpha >\alpha _0$ we have $x_\alpha \in -s_\alpha + \oplam(t_\alpha+W)$. This shows that
\[
x_\alpha + s_\alpha \in \oplam(t_\alpha+W) \,.
\]

Next, note that $(x_\alpha + s_\alpha -x-s, t_\alpha- t) +\cL \to (0,0)+\cL$.

By the uniform discreteness of $\cL$ in $G\times H$ we can then find some elements $(l_\alpha,l^\star_\alpha) \in \cL$ such that, in $G \times H$ we have
\[
\lim_\alpha (x_\alpha + s_\alpha -x-s+l_\alpha, t_\alpha- l^\star_\alpha)=(0,0) \,.
\]

Let $U', V'$ be neighbourhoods of $0$ in $G$ and $H$ respectively.

Then, there exists some $\alpha$ such that
\[
(x_\alpha + s_\alpha -x-s+l_\alpha, t_\alpha-t+ l^\star_\alpha) \in U' \times V' \,.
\]

It follows that,
\begin{align*}
x+s &  \in x_\alpha + s_\alpha+l_\alpha - U \subseteq \oplam(t_\alpha+W)+l_\alpha -U=\oplam(t_\alpha+l_\alpha^\star +W) -U\\
\end{align*}
and
\[
t_\alpha+ l^\star_\alpha  \in t+V' \,.
\]

This gives,
\[
x+s \in \oplam(t+W+V') -U' \,.
\]

Since $U',V'$ are arbitrary and $W$ is compact, the claim follows.

\end{proof}

\begin{theorem}\label{thm:torus par 1} Let $(G, H, \cL)$ be a CPS and $W \subseteq H$ be a compact window. Then,
\begin{align*}
\YY_{app}&=\YY_{b,m}=\YY_{m}= \{ \Gamma \in \extW : \dens_{\cA}(\Gamma)= \dens(\cL) \theta_H(W) \} \,.
\end{align*}
Moreover, for each $\Gamma \in \YY_{app}$ there exists some $(s,t)+\cL \in \TT$ with the following properties
\begin{itemize}
\item{} $\Gamma \subseteq -s+\oplam(t+W)$,
\item{} $-s+\oplam(t+W) \in \YY_{app}$,
\item{}
$\dens((-s+\oplam(t+W)) \backslash \Gamma) =0$,
\item{} \eqref{eq9} holds for $\Gamma$ and $-s+\oplam(t+W)$ for $(s,t)+\cL \in \TT$.
\end{itemize}
In particular,
\[
\pi(\Gamma)= (s,t)+\cL+ \{ 0\} \times H_W \,.
\]
\end{theorem}
\begin{proof}

We already know that
\[
\YY_{app} \subseteq \YY_{b,m} \subseteq \YY_{m} \subseteq \{ \Gamma \in \extW : \dens_{\cA}(\Gamma)= \dens(\cL) \theta_H(W) \} \,.
\]

Now, let $\Gamma \in \extW$ be so that  $\dens_{\cA}(\Gamma)= \dens(\cL) \theta_H(W)$.

By Lemma~\ref{lem torus} there exists some $(s,t) \in \TT$ such that
\[
\Gamma \subseteq -s+\oplam(t+W) \,.
\]
We then have
\begin{align*}
\dens(\cL) \theta_H(W)&= \dens_{\cA}(\Gamma) \leq \liminf_k \frac{\delta_{-s+\oplam(t+W)}(A_k)}{|A_k|} \\
&\leq \limsup_k \frac{\delta_{-s+\oplam(t+W)}(A_k)}{|A_k|} \leq \dens(\cL) \theta_H(t+W)
\end{align*}
with the last inequality following from \cite{HR}.

This implies that the density of $-s+\oplam(t+W)$ exists with respect to $\cA$ and
\[
\dens(\cL) \theta_H(W)= \dens_{\cA}(\Gamma) = \dens_{\cA}(-s+\oplam(t+W)) \,.
\]
Since $\Gamma \subseteq -s+\oplam(t+W)$ we also have
\begin{displaymath}
\dens((-s+\oplam(t+W)) \backslash \Gamma) =0 \,.
\end{displaymath}

Finally, we show that\eqref{eq9} holds for $\Gamma$ and $-s+\oplam(t+W)$ for $(s,t)+\cL \in \TT$.

We first show it holds for $-s+\oplam(t+W)$. For this, note that since $W$ is compact, we can chose $K_n =W$ and $W \subseteq U_n$. Then, the choice of $g_n$ implies that
$g_n =1$ on $W$.
Then $T_{-s} \omega_{T_t g_n} \geq \delta_{-s+\oplam(t+W)}$ and hence

\begin{align*}
&\uM_{\cA}(\left| T_{-s}\omega_{T_t g_n} - \delta_{-s+\oplam(t+W)}\right|)  = \uM_{\cA}(T_{-s}\omega_{T_t g_n}  - \delta_{-s+\oplam(t+W)}) \\
 &= \dens(\cL) \int_{H} T_t g_n(t) \dd  y - \dens(\cL) \theta(t+W)  \leq \dens(\cL) \left( \theta_H(U_n)- \theta_H(K_n) \right) < \frac{\dens{\cL)}}{n} \,.
\end{align*}

Therefore
\begin{displaymath}
\lim_n \uM_{\cA}(\left| T_{-s}\omega_{T_t g_n} - \delta_{-s+\oplam(t+W)}\right|)=0 \,.
\end{displaymath}

The claim follows now from
\begin{displaymath}
\uM_{\cA}(\left| \delta_{\Gamma} - \delta_{-s+\oplam(t+W)}\right|)=0 \,.
\end{displaymath}
\end{proof}

As consequences we get:

\begin{proposition}\label{Prop1} Let $(G, H, \cL)$ be a CPS and $W \subseteq H$ be a compact window, and $\Gamma \in \YY_{app}$. Then, there exists some $(s,t)+\cL \in \TT$ such that
\[
\Gamma \subseteq -s+\oplam(t+W) \,.
\]
Moreover, if $\Gamma \subseteq -s'+\oplam(t'+W)$ then
$$(s,t)+\cL + 0\times H_{W}= (s',t')+\cL + 0\times H_{W} \,.$$
\end{proposition}
\begin{proof}

The existence of $(s,t)$ follow from Lemma~\ref{lem torus}. Now, let $(s,t)$ and $(s',t')$ be so that
\begin{align*}
\Gamma &\subseteq -s+\oplam(t+W) \\
\Gamma &\subseteq -s'+\oplam(t'+W) \,.
\end{align*}

Note that if $\Gamma = \emptyset$ we must have $\theta(W)=0$ and the claim is trivial. So we can assume without loss of generality that $\Gamma \neq \emptyset$. Then
\begin{align*}
  -s+\oplam(t+W) \cap  -s'+\oplam(t'+W)  & \neq \emptyset \Longrightarrow \\
  -s+s' + \oplam(t+W) \cap \oplam(W)  &\neq \emptyset \Longrightarrow \\
  s-s' &\in L =\pi_G(\cL) \,.
\end{align*}
This shows that there exists some $(x,x^\star) \in \cL$ and $h \in H$ such that
\[
(s',t')=(s,t)+(x,x^\star)+(0,h) \,.
\]
Note that
\[
-s'+\oplam(t'+W)= -s +\oplam(t+h+W) \,.
\]
We complete the proof by showing that $h \in H_W$.

Indeed,
\[
\Gamma \subseteq \left(-s+\oplam(t+W)\right) \cap \left(-s+\oplam(t+h+W)\right)=-s+\oplam((t+W) \cap (t+h+W))
\]
which implies that
\begin{align*}
\dens(\cL) \theta_H(W)&= \dens(\Gamma) \leq \overline{\dens}(-s+\oplam((t+W) \cap (t+h+W))) \\
& \leq \dens(\cL) \theta_H((t+W) \cap (t+h+W))=  \dens(\cL) \theta_H(W \cap (h+W)) \,.
\end{align*}
This implies that $\theta_H(W)=\theta_H(W \cap (h+W))$ from where it follows that $h \in H_W$.

\end{proof}

By combining Theorem~\ref{thm:torus par 1} and Proposition~\ref{Prop1} we get the following equivalent characterisation for $\pi$.

\begin{corollary}\label{cor 84} Let $\theta(W)> 0, \Gamma \in \YY_{app}$ and $(s,t) \in G \times H$. Then, the following are equivalent:

\begin{itemize}
  \item[(i)] $\pi(\Gamma)= (s,t)+\cL+ \{ 0\} \times H_W$.
  \item[(ii)] There exists some $(s',t') \in G \times H$ such that $(s,t)+\cL+ \{ 0\} \times H_W=(s',t')+\cL+ \{ 0\} \times H_W$ and
  \[
  \Gamma \subseteq -s'+\oplam(t'+W) \,.
  \]
  \item[(iii)] $s+\Gamma \subseteq L$ and
\[
\theta_H\left( (\overline{ \{ (s+\Gamma)^\star\}} \cap (t+W) \right) =\theta_H(W) \,.
\]
\end{itemize}
\end{corollary}
\begin{proof}
By Lemma~\ref{lem torus} there exists some $(s',t') \in G \times H$ such that
  \[
  \Gamma \subseteq -s'+\oplam(t'+W) \,.
  \]
In particular, $\pi(\Gamma)= (s',t')+\cL+ \{ 0\} \times H_W$.

The equivalence \textbf{(i)} $\Longleftarrow$ \textbf{(iii)}  follows from here.

\smallskip
\textbf{(ii)} $\Longleftrightarrow$ \textbf{(iii)}

Since
\[
\Gamma \subseteq -s'+\oplam(t'+W)
\]
we have
\begin{equation}\label{eq78}
s'+\Gamma \subseteq \oplam(t'+W) \,.
\end{equation}
Next, $(s,t)+\cL+ \{ 0\} \times H_W=(s',t')+\cL+ \{ 0\} \times H_W$ gives that there exists some $x \in L$ and $h \in H_W$ such that
\[
(s',t')=(s,t)+(x,x^\star)+(0,h) \,.
\]
Therefore, \eqref{eq78} becomes
\[
s+x+\Gamma \subseteq \oplam(t+x^\star+h+W)
\]
and hence
\[
s+\Gamma \subseteq -x+\oplam(t+x^\star+h+W) =  \oplam(t+h+W) \,.
\]
This implies that $s+ \Gamma \subseteq L$ and
\[
(s+\Gamma)^\star \subseteq t+h+W
\]
and hence, since $t+h+W$ is compact we have
\[
\overline{(s+\Gamma)^\star} \subseteq t+h+W \,.
\]
Let
\[
W':= \overline{(s+\Gamma)^\star} \subseteq t+h+W \,.
\]
Then,
\[
s+ \Gamma \subseteq \oplam(W')
\]
and hence, by compactness of $W'$ have
\[
\udens_{\cA}{s+\Gamma} \leq \dens(\cL) \theta(W') \,.
\]

Since $\Gamma \in \YY_{app}$ by Theorem~\ref{thm:torus par 1} we have
\[
\dens_{\cA}{\Gamma} = \dens(\cL) \theta(W)= \dens(\cL) \theta(t+h+W)  \,.
\]
Next, using $W' \subseteq t+h+W$ we get
\begin{align*}
\dens(\cL) \theta(t+h+W)&=\dens_{\cA}{\Gamma}=\dens_{\cA}{s+\Gamma} \leq \udens_{\cA}(s+\Gamma) \\
&\leq \dens(\cL) \theta(W')\leq \dens(\cL) \theta(t+h+W)
\end{align*}
and hence
\[
\theta(W')=  \theta(t+h+W) \,.
\]
Finally, since $h \in H_W$ we have
\[
\theta_{H}( (t+W) \Delta (t+h+W))=0 \,.
\]
Since $W' \subseteq  t+h+W$ we have $W' \backslash t+W \subseteq (t+h+W) \Delta (t+W)$ and hence
\[
\theta_{W}(W' \backslash t+W )=0 \,.
\]
Therefore,
\begin{align*}
\theta(W' \cap (t+W))=\theta(W') - \theta(W' \backslash (t+W)) = \theta_H(W')= \theta_H(t+h+W)=\theta_H(W) \,.
\end{align*}

\medskip

\textbf{(iii)} $\Longrightarrow$ \textbf{(ii)}

Set again
\[
W'= \overline{ \{ (s+\Gamma)^\star\}} \,.
\]

Then
\[
s+\Gamma =(s-s')+s'+\Gamma \subseteq  (s-s')+\oplam(t'+W) \,.
\]

By (iii) we have $s+\Gamma \subseteq L$. Moreover, $s'+\Gamma \subseteq \oplam(t'+W) \subseteq L$. From here it follows that $(s-s') \in L$.
Therefore
\[
s+\Gamma =(s-s')+s'+\Gamma \subseteq  (s-s')+\oplam(t'+W)\stackrel{\eqref{eq:translate oplam(W)}}{=}\oplam((s-s')^\star+t'+W) \,.
\]
This implies that $(s+\Gamma)^\star \subseteq (s'-s)^\star+t'+W$, and hence, since $W$ is compact, we have
\[
W' \subseteq (s-s')^\star+t'+W \,.
\]

Looking again at the densities we get
\begin{align*}
\dens(\cL) \theta_H((s-'s)^\star+W)&=\dens(\cL) \theta_H(W)= \dens(\Gamma) \leq \dens( \oplam(W')) \\
&\leq \dens(\cL) \theta_H(W') \leq \dens(\cL) \theta_H((s-s')^\star+t'+W)
\end{align*}
and hence
\[
\theta_H(W')=\theta_H((s-s')^\star+t'+W) \,.
\]

Finally by (iii) we have
\[
\theta_H\left( W' \cap (t+W) \right) =\theta_H(W) \,.
\]

Then
\begin{align*}
\theta_H(t+W \Delta ((s-s')^\star+t'+W)) & =\theta_H((t+W) \Delta W' \Delta W' \Delta   ((s-s')^\star+t'+W)) \\
&\leq  \theta_H(\bigl(W \Delta W' \bigr) \cup \bigl( W' \Delta   ((s'-s)^\star+t'+W)\bigr)) \\
&\leq  \theta_H\bigl(W \Delta W' \bigr) + \theta_H \bigl( W' \Delta   ((s-s')^\star+t'+W)\bigr) \\
&=0 +\theta_H \bigl( W' \Delta   ((s-s')^\star+t'+W)\bigr) =0
\end{align*}
with the last equality following from $W \subseteq (s-s')^\star+t'+W$ and $\theta_H(W')=\theta_H((s-s')^\star+t'+W)$.

This implies that $(s-s')^\star+t' -t\in H_W$.

Therefore
\begin{align*}
\pi(\Gamma)&=(s',t')+\cL+ \{ 0\} \times H_W\\
&= (s,t)+(s'-s,(s'-s)^\star)+(0,t'-t+(s-s')^\star)+ \cL+ \{ 0\} \times H_W \\
&=(s,t)+\cL+ \{ 0\} \times H_W \,.
\end{align*}

This proves the claim.

\end{proof}

Combining all the previous results, we get:

\begin{theorem}\label{thm:8.5} Let $(G, H, \cL)$ be a CPS and $W \subseteq H$ be a compact window of positive measure. Then, the mapping $\pi : \YY_{app} \to \TT$ has the following properties:
\begin{itemize}
  \item[(a)] For $(s,t)+\cL \in \TT$ we have $ (s,t)+\cL \in \mbox{Im}(\pi)$ if and only if $-s+\oplam(t+W)$ is a maximal density weak model set.
  \item[(b)] $\theta_{\TT_{red}}(\mbox{Im}(\pi)) = 1$.
  \item[(c)] $\pi$ is continuous.
  \item[(d)] $\pi$ induces a Borel factor from $(\extW, G, m)$ to $(\TT_{red}, G)$.
\end{itemize}
\end{theorem}
\begin{proof}
\textbf{(a)} Follows from Theorem~\ref{thm:torus par 1}.

\textbf{(b)} Follows from Theorem~\ref{thm:main}.

\textbf{(c)}

We are following here the idea of \cite[Lemma~4.2]{Martin2}.

Recall here that $G$ is second countable.

By the compactness of $W$, there exists some $r_0>0$ such that all elements $\Gamma \in \extW$ are $r_0$-uniformly discrete.
Since $G$ is $\sigma$ compact, we can find a sequence $K_m$ of compact sets such that $0 \in K_1^\circ$, $K_m \subseteq K_{m+1}^\circ$ and
\[
G =\bigcup_m K_m \,.
\]

Let $\Gamma_n, \Gamma \in \YY_{app}$ be so that $\Gamma_n \to \Gamma$. Since the topology on $\extB$ is metrisable, we need to show that $\pi(\Gamma_n) \to \pi(\Gamma)$.

By Lemma~\ref{l4}, by eventually replacing $\Gamma_n, \Gamma$ with $-x+\Gamma_n$ and $-x+\Gamma$ for some $x \in \Gamma$, we can assume without loss of generality that
$0 \in \Gamma$.

Next, fix some $r>0$ such that $2r <r_0$. Since $\Gamma_n \to \Gamma$, there exists some $N$ such that, for all $n>N$ we have $0 \in \Gamma_n+B_r(0)$.

Therefore, for each $n >N$ there exists some $x_n \in B_r(0)\cap \Gamma_n$, which must be unique by the $r_0$-uniform discreteness of $\Gamma_n$.
By eventually erasing the first few $\Gamma_n$ we can assume without loss of generality that such $x_n$ exists for all $n$.

It is easy to see that
\[
\lim_n x_n =0
\]
in $G$.

Indeed, let $\eps >0$ and pick $f_\eps$ such that $f_\eps \geq 0, \supp(f_\eps) \subseteq B_{\eps}(0) \cap B_r(0)$ and $f_\eps0(0)=1$.

Since $\Gamma_n \to \Gamma$ in the local rubber topology, we have
\[
\delta_{\Gamma}(f_\eps)= \lim_n \delta_{\Gamma_n}(f_\eps) \,.
\]

Therefore, there exists some $N$ such that, for all $n>N$ we have
\[
\left| \delta_{\Gamma}(f_\eps)- \delta_{\Gamma_n}(f_\eps) \right| < \frac{1}{2} \,.
\]
Since $0 \in \Gamma$ we have $\delta_{\Gamma}(f_\eps) \geq f_{\eps}(0)=1$ and hence $\delta_{\Gamma_n}(f_\eps) \geq \frac{1}{2}$.

This implies that
\[
\Gamma_n \cap \supp(f_\eps) \neq \emptyset \,.
\]
Thus, we get
\[
\emptyset \neq \Gamma_n \cap \supp(f_\eps) \subseteq \Gamma_n \cap B_r(0) = \{ x_ n\}
\]
and hence
\[
\Gamma_n \cap \supp(f_\eps)= \{ x_ n\} \,.
\]
It follows that $x_n \subseteq \supp(f_\eps) \subseteq B_{\eps}(0)$.

This shows that for all $n >N$ we have $x_n \in B_\eps(0)$. This proves that
\[
\lim_n x_n =0 \,.
\]

\medskip

Next, define
\[
\Lambda_n := -x_n +\Gamma_n \,.
\]

Then, $\Lambda_n \in \YY_{app}, 0 \in \Lambda_n, \Lambda_n \to \Gamma$ and
\[
\pi(\Gamma_n) -\pi(\Lambda_n)= (x_n ,0) +\cL + \{ 0 \} \times H_W \,.
\]

Therefore, to complete the proof, it suffices to show that $\pi(\Lambda_n) \to \pi(\Gamma)$.

For each $\Lambda_n$ pick some $(s_n',t_n') \in \cL$ such that
\[
\Lambda_n \subseteq -s_n'+\oplam(t_n'+W) \,,
\]
and pick some $(s',t') \in \cL$ such that
\[
\Gamma \subseteq -s'+\oplam(t'+W) \,.
\]

Since $0 \in \Lambda_n$ we have $s_n' \in L$ and hence
\[
\Lambda_n \subseteq \oplam(s_n'+t_n'+W) \,.
\]
Same way
\[
\Gamma \subseteq \oplam(s'+t'+W) \,.
\]

Setting
\begin{align*}
  t_n &:=s_n'+t_n'  \\
  t&:=s'+t'
\end{align*}
we get
\begin{align*}
  \Lambda_n &\subseteq \oplam(t_n+W) \,, \\
 \Lambda & \subseteq \oplam(t+W) \,,
\end{align*}
and
\begin{align*}
  \pi(\Lambda_n)&= (0, t_n)+\cL+\{ 0\} \times H_W \,, \\
  \pi(\Gamma)&= (0, t)+\cL+\{ 0\} \times H_W  \,.\\
\end{align*}

To complete the proof we show that $t_n +H_W \to t+H_W$ in $H/H_{W}$.

Let $V$ be an  neighbourhood of $t+H_W$ in $H/H_W$. Then, there exists a neighbourhood $t\in U$ such that
\[
U_W:= U+H_{W} \subseteq H
\]
satisfies
\[
p(U_W) \subseteq V \,,
\]
where $p : H \to H/H_W$ is the canonical projection.

Note that
\[
U_W =\bigcup_{ h\in H_W} h+U
\]
is open in $H$.

Next, as in \cite[Lemma~4.1]{Martin2}, let us look at the set
\[
J:= \bigcup_{x \in \Gamma} (x^\star -W) \,.
\]

We claim
\[
t \in J \subseteq t+U_W \,.
\]

Indeed, for all $x \in \Gamma \subseteq \oplam(t+W)$ we have $x^\star \in t+W$ and hence $t \in x^\star- W$. This gives $t \in J$.

Fix some $h \in J$.

Then, for all $x \in \Gamma$ we have $h \in x^\star -W$ and hence
\[
x^\star \in h+W \,.
\]
This gives
\[
\Gamma \subseteq \oplam(h+W) \,.
\]
Then, by Proposition~\ref{Prop1} we have
\[
(0,h) +\cL+\{ 0\} \times H_W  = (0,t)+\cL+\{ 0\} \times H_W \,.
\]
Since $\cL \cap \{ 0\} \times H = \{ (0,0) \}$ this implies that $h \in t+H_W$ which completes the claim.

Now, we employ the Schlottmann idea \cite[Lemma~4.1]{Martin2}: Since

\[
J:= \bigcup_{x \in \Gamma} (x^\star -W) \subseteq t+H_W
\]
we have
\[
\bigcup_{x \in \Gamma} \left( (x^\star -W) \backslash (t+H_W) \right) =\emptyset \,.
\]

Since $\left( (x^\star -W) \backslash (t+H_W) \right)$ is compact for each $n$, there exists a finite set $F \subseteq \Gamma$ such that

\[
\bigcup_{x \in F} \left( (x^\star -W) \backslash (t+H_W) \right) =\emptyset \,,
\]
and hence
\[
\bigcup_{x \in F}(x^\star -W) \subseteq (t+H_W) \,.
\]

Now, since $\Gamma_n \to \Gamma, 0 \in \Gamma_n, 0 \in \Gamma$ and $\Gamma_n- \Gamma_n, \Gamma-\Gamma$ are equi uniformly discrete, it follows immediately that there exists some $N$ such that, for all $n >N$ we have $F \subseteq \Gamma_n$.

This implies that, for all $n>N$ we have
\[
\bigcup_{x \in \Gamma_n}(x^\star -W) \subseteq \bigcup_{x \in F}(x^\star -W) \subseteq (t+H_W) \,.
\]

Pick an arbitrary $n>N$. Then, for all $x \in \Gamma_n$ we have $x^\star \in t_n+W$ and hence $t_n \in x^\star -W$. This gives
\[
t_n \in \bigcup_{x \in \Gamma_n}(x^\star -W) \subseteq \bigcup_{x \in F}(x^\star -W) \subseteq (t+H_W) \forall n >N
\]
and hence $t_n+H_W \in U$ for all $n>N$.

This shows that
\[
\lim_n t_n+H_W = t+H_W \mbox{ in } H/H_W \,,
\]
which completes the proof of (c).

\textbf{(d)}

The set $\YY_{app}$ has full measure in $\extW$ and $\pi(\YY_{app})$ has full measure in $\TT$. Since $\pi : \YY_{app} \to \TT_{red}$ is continuous, it is Borel.

Finally, the pushforward of $m$ through $\pi_{app}$ is $G$ invariant. Since the orbit of $G$ is dense in $\TT$, it follows that the push forward of $m$ is a Haar measure on $\TT$ and hence the probability Haar measure.
\end{proof}

\begin{remark} When $\theta(W)=0$ then the zero measure $ 0 \in \extW$ and $\m=\delta_{0}$. Moreover, in this case the previous theorem trivially holds if we define $\TT_{red}= (G \times H)/ \overline{\cL+\{ 0\} \times H_W}$, in which case it is the trivial group.
\end{remark}

Next, by Pontryagin Duality we have a canonical isomorphism between $\widehat{\TT_{red}}$ and
\begin{align*}
\SSS:= \mbox{Ann} (\cL +\{ 0 \} \times H_W) = \{ (\chi, \chi^\star) \in \cL^0 | \chi^\star(x)=1 \forall x \in H_W \} \\
\end{align*}
given by
\[
(\chi, \chi^\star) \left( (s,t)+\cL+ \{ 0\} \times H_W \right)\to  \chi(s) \chi^\star(t) \,.
\]

Next, since $1_{W} \in L^1(H)$ is $H_W$ periodic, the function $\widecheck{1_W}$ is supported on $\{ \psi \in H : \psi(x)=1 \forall x \in H_W\}$. It follows immediately that
\begin{equation}\label{eq 9}
\{ \chi : \widehat{\gamma}(\{ \chi \}) \neq 0 \} \subseteq \{ \chi : \exists  \psi \in \widehat{H} : (\chi, \psi) \in \SSS \} \,.
\end{equation}

Combining \eqref{eq 9} with the results in this section we get the following summary.

\begin{theorem}\label{Main 3} Let $(G,H, \cL)$ be a CPS and $W$ a compact window. Then
\begin{itemize}
  \item[(a)]
\begin{align*}
\YY_{app}&=\YY_{b,\m}=\YY_{m}= \{ \Gamma \in \extW : \dens_{\cA}(\Gamma)= \dens(\cL) \theta_H(W) \} \,.
\end{align*}
\item[(b)]   $\m(Y_{app})=1$.
\item[(c)] For $(s,t)+\cL \in \TT$ we have $ (s,t)+\cL \in \mbox{Im}(\pi)$ if and only if $-s+\oplam(t+W)$ is a maximal density weak model set.
\item[(d)] $\theta_{\TT_{red}}(\mbox{Im}(\pi)) = 1$.
\item[(e)] The mapping $\pi : \YY_{app} \to \TT_{red}$ is continuous.
\item[(f)] For $\Gamma \in \YY_{app}$ we have $\pi(\Gamma)= (s,t)+\cL + \{0 \} \times H_W$ if and only if there exists $(s',t') \in  (s,t)+\cL + \{0 \} \times H_W$
such that
\[
\Gamma \subseteq -s'+\oplam(t+W) \,.
\]
\item[(g)] $\pi$ induces a Borel factor from $(\extW, G, \m)$ to $(\TT_{red}, G)$.
\item[(h)] For each $(\chi, \chi^\star) \in \SSS$ the mapping $f_\chi( \Gamma):= (\chi, \chi^\star) \left(  \pi(\Gamma) \right)$ is an eigenfunction for $\chi$
  which is continuous on $\YY_{app}$.
\item[(i)] $\pi$ induces an isometric isomorphism
\[
\pi: L^2(\TT_{red}, \theta_{\TT_{red}}) \to L^2(\extW, \m) \,.
\]
\end{itemize}
\end{theorem}
\begin{proof}
\textbf{(a)} Follows from Theorem~\ref{thm:torus par 1}.

\medskip
\textbf{(b)} Follows from (a) and Thm.~\ref{thm:main 2}.

\medskip
\textbf{(c),(d),(e), (g)} Follow from Thm.~\ref{thm:8.5}.

\medskip
\textbf{(f)} Follows from Cor.~\ref{cor 84}.

\medskip
\textbf{(h)} Let $(\chi, \chi^\star) \in \SSS$ and define $f_\chi : \YY_{app} \to \CC$
\[
f_\chi:=(\chi, \chi^\star) \circ \pi \,.
\]

Then, as composition of continuous functions, $f_\chi$ is continuous. Moreover, for all $s \in G$ and $\Gamma \in \YY_{app}$ we have
\[
f_\chi(s+ \Gamma)=(\chi, \chi^\star)(\pi(s+\Gamma))=(\chi, \chi^\star)((s,0)+\pi(\Gamma))=\chi(s) \cdot  \chi^\star(0) f_\chi(\Gamma) \,.
\]
This shows that $f_\chi$ satisfies
\[
f_\chi(s+\Gamma)= \chi(s) f_\chi(\Gamma)
\]
for all $\Gamma \in \YY_{app}$.

Setting $f_{\chi}=1$ outside $\YY_{app}$, and using the fact that $\YY_{app}$ is $G$-invariant and $\m(\YY_{app})=1$ gives that $f_{\chi}$ is an eigenfunction for $\chi$.

\medskip
\textbf{(i)} Since $\pi$ induces a Borel factor from $(\extW, G, \m)$ to $(\TT_{red}, G)$ we get that
\[
\int_{\TT_{red}} g (x) \dd \theta_{\TT_{red}}(x) = \int_{\extW} g \circ \pi (\Gamma ) \dd \m(\Gamma) \,,
\]
for all $g \in C(\TT_{red})$.

In particular, for all $g \in C(X)$ we get

\[
\| g \|_{L^2(\TT_{red}, \theta_{\TT_{red}})} = \|  g \circ \pi \|_{L^2(\extW, \m)}  \,.
\]

Via a standard density argument, the mapping $\pi$ has an unique extension to an isometric mapping
\[
\pi : L^2(\TT_{red}, \theta_{\TT_{red}}) \to L^2(\extW, \m) \,.
\]

As an isometry, this mapping is one to one. We complete the proof by arguing that this is onto.

Since $\m$ has pure point dynamical system, the span of eigenfunctions is dense in $L^2(\extW, \m)$.  Moreover, the spectrum of $L^2(\extW, \m)$ is generated by
\[
\{ \chi : \widehat{\gamma}(\{ \chi \}) \neq 0 \} =: S \,.
\]

Now, by \eqref{eq 9}, for each $\chi \in S$ there exists some $\chi^\star$ such that $(\chi, \chi^\star) \in \SSS$. Therefore, by (f) there is an eigenfunction
\[
f_\chi =  (\chi, \chi^\star) \circ \pi
\]
which by construction belongs to $\pi (C(\TT_{red})) \subseteq \mbox{Im}(\pi)$.

Now, let $\psi$ be any eigenvalue of $L^2(\extW, \m)$. Then, there exists some $\chi_1,..., \chi_{k}, \chi_{k+1},..., \chi_{j} \in S$ such that
\[
\psi=\chi_1 \cdot \chi_2 \cdot...\cdot \chi_{k} \cdot \overline{\chi_{k+1}} \cdot \ldots \cdot \overline{\chi_j} \,.
\]

It follows immediately that
\[
f_{\chi_1} \cdot f_{\chi_2} \cdot \ldots \cdot f_{\chi_k} \cdot \overline{f_{\chi_{k+1}}} \cdot \ldots  \cdot \overline{f_{\chi_{j}}}
\]
is an eigenfunction for $\psi$, which is the image under $\pi$ of
\[
(\chi_1, \chi_1^\star) \cdot (\chi_2, \chi_2^\star)  \cdot...\cdot (\chi_k, \chi_k^\star)  \cdot \overline{(\chi_{k+1}, \chi_{k+1}^\star) } \cdot \ldots \cdot \overline{(\chi_j, \chi_j^\star) }  \in C(\TT_{red}) \,.
\]
Therefore, the range of $\pi$ contains all the eigenfunctions of $L^2(\extW, \m)$, and hence is dense in $L^2(\extW, \m)$.

Since $\pi : L^2(\TT_{red}, \theta_{\TT_{red}}) \to L^2(\extW, \m)$ is an isometry, the range is a Banach space and hence closed in $L^2(\extW, \m)$. This shows that
\[
\pi (L^2(\TT_{red}, \theta_{\TT_{red}})) = L^2(\extW, \m) \,,
\]
which completes the proof.
\end{proof}

\begin{remark} It follows from Theorem~\ref{Main 3} that $\TT_{red}$ is the Maximum Generic Equicontinuous Factor (MGEF) of $\extW$ (see \cite{Kel1} for definition)).
\end{remark}

\section{Model sets with open precompact windows}

In this section we show that for model sets with open precompact windows we get similar results to the case of weak model sets. The proof techniques are similar, the
only difference is that the inclusion in Lemma~\ref{lem torus} gets reversed.

To keep the length of the paper short we will only prove the results where the proof is not identical to the compact window case, and leave all the other proofs as homework for the reader.

\begin{lemma} Let $(G, H, \cL)$ be a CPS, $U \subseteq H$ be an open precompact window and let $\Gamma \in \extU$. Then, there exists some $(s,t)+\cL \in \TT$ such that
\[
\Gamma \supseteq -s+\oplam(t+U) \,.
\]
\end{lemma}
\begin{proof}

Let $(s_n, t_n) +\cL \in \TT$ be so that $-s_n + \oplam(t_n+U) \to \Gamma$. We know that such elements exist. By compactness, the sequence $(s_n,t_n)+\cL$ has a convergent subnet $(s_\alpha, t_\alpha)+\cL$ which converges to some $(s,t)+\cL \in \TT$.

We show that
\[
\Gamma \supseteq -s+\oplam(t+U) \,.
\]

Let $x \in  -s+\oplam(t+U)$ be arbitrary. Then $x+s \in \oplam(t+U)$. This implies that $x+s \in L$ and $(x+s)^\star \in t+U$.

Therefore we have
\[
\lim_\alpha (x+s)^\star-t_{\alpha}=(x+s)^\star-t \in U \,.
\]
Since $U$ is open, there exists some $\alpha_0$ such that, for all $\alpha > \alpha_0$ we have
\[
(x+s)^\star-t_{\alpha}-t_{\alpha} \in U \,.
\]
This gives that
\[
x+s \in \oplam(t_{\alpha}+U)
\]
for all $\alpha >\alpha_0$ and hence
\[
x+s-s_{\alpha} \in -s_{\alpha}+\oplam(t_{\alpha}+U) \,.
\]

By construction $-s_{\alpha}+\oplam(t_{\alpha}+U)$ is a subnet of $-s_n + \oplam(t_n+U)$ and hence
\begin{equation}\label{eq:8}
\Gamma = \lim_{\alpha} -s_{\alpha}+\oplam(t_{\alpha}+U) \,.
\end{equation}

Now, pick $0 \in V'$ open so that $\Gamma$ is $V'$ uniformly discrete, and let $0 \in V=-V$ be open so that $V-V \subset V'$.

Pick some $f \in \Cc(G)$ such that $f(x)=1, f \geq 0$ and $\supp(f) \subseteq x+V$. By \eqref{eq:8} we have
\begin{displaymath}
\delta_{\Gamma}(f) = \lim_{\alpha} \delta_{-s_{\alpha}+\oplam(t_{\alpha}+U)}(f) \,.
\end{displaymath}

Now, for all $\alpha >\alpha_0$ we have $x+s-s_{\alpha} \in -s_{\alpha}+\oplam(t_{\alpha}+U)$ and hence $\delta_{-s_{\alpha}+\oplam(t_{\alpha}+U)}(f)
\geq f(x+s-s_{\alpha})$.

This gives
\begin{displaymath}
\delta_{\Gamma}(f) \geq  \lim_{\alpha}f(x+s-s_{\alpha}) = f(x)=1 \,.
\end{displaymath}

This implies that
\[
\Gamma \cap \supp(f) \neq \emptyset \,.
\]

Therefore, there exists some $y \in \Gamma \cap (x+V)$, which is unique by the $V'$ uniform discreteness of $\Gamma$.

We show that $x=y$, which implies that $x \in \Gamma$. This proves the claim.

Assume by contradiction that $x \neq y$. Then there exists some open neighbourhood $V''$ of $0$ such that $V'' \subseteq V$ and $y-x \notin V''$. We now repeat the above argument with $V''$ instead of $V$:

Pick some $f'' \in \Cc(G)$ such that $f''(x)=1, f'' \geq 0$ and $\supp(f'') \subseteq x+V''$. By \eqref{eq:8} we have
\begin{displaymath}
\delta_{\Gamma}(f'') = \lim_{\alpha} \delta_{-s_{\alpha}+\oplam(t_{\alpha}+U)}(f'') \geq \lim_{\alpha}f''(x+s-s_{\alpha}) = f''(x)=1 \,.
\end{displaymath}

This implies that
\[
\Gamma \cap \supp(f'') \neq \emptyset \,.
\]
It follows that
\[
\emptyset \neq \Gamma \cap \supp(f'') \subseteq \Gamma \cap (x+V'') \subseteq \Gamma \cap (x+V) = \{ y\} \,.
\]
This yields
\[
\Gamma \cap (x+V'') = \{ y\}
\]
and hence
\[
y \in x +V''' \,.
\]

But this contradicts the fact that $y-x \notin V''$.
\end{proof}

The proof of the following result is similar to Theorem~\ref{thm:torus par 1} and we skip it.

\begin{theorem}\label{thm:torus par 2} Let $(G, H, \cL)$ be a CPS and $U \subseteq H$ be an open precompact window. Then,
\begin{align*}
\YY_{app}&=\YY_{b,\m}=\YY_{\m}= \{ \Gamma \in \extU : \dens_{\cA}(\Gamma)= \dens(\cL) \theta_H(U) \} \,.
\end{align*}
Moreover, for each $\Gamma \in \YY_{app}$ there exists some $(s,t)+\cL \in \TT$ with the following properties
\begin{itemize}
\item{} $\Gamma \supseteq -s+\oplam(t+W)$,
\item{} $-s+\oplam(t+U) \in \YY_{app}$,
\item{}
$\dens((-s+\oplam(t+U)) \backslash \Gamma) =0$,
\item{} \eqref{eq9} holds for $\Gamma$ and $-s+\oplam(t+U)$ for $(s,t)+\cL \in \TT$.
\end{itemize}
In particular,
\[
\pi(\Gamma)= (s,t)+\cL+ \{ 0\} \times H_U \,.
\]
\end{theorem}

As before, the following is an immediate consequence.

\begin{proposition}\label{Prop1'} Let $(G, H, \cL)$ be a CPS, $U \subseteq H$ be an open precompact window. and $\Gamma \in \YY_{app}$. Then, there exists some $(s,t)+\cL \in \TT$ such that
\[
\Gamma \supseteq -s+\oplam(t+U) \,.
\]
Moreover, if $\Gamma \supseteq -s'+\oplam(t'+U)$ then
$$(s,t)+\cL + 0\times H_{U}= (s',t')+\cL + 0\times H_{U} \,.$$
\end{proposition}

By combining Theorem~\ref{thm:torus par 2} and Proposition~\ref{Prop1'} we get the following equivalent characterisation for $\pi$.

\begin{corollary} Let $\theta(U)> 0, \Gamma \in \YY_{app}$ and $(s,t) \in G \times H$. Then, the following are equivalent:

\begin{itemize}
  \item[(i)] $\pi(\Gamma)= (s,t)+\cL+ \{ 0\} \times H_U$.
  \item[(ii)] There exists some $(s',t') \in G \times H$ such that $(s,t)+\cL+ \{ 0\} \times H_U=(s',t')+\cL+ \{ 0\} \times H_U$ and
  \[
  \Gamma \supseteq -s'+\oplam(t'+U) \,.
  \]
  \item[(iii)] $s+\Gamma \subseteq L$ and
\[
\theta_H\left( (\overline{ \{ (s+\Gamma)^\star\}} \cap (t+U) \right) =\theta_H(U) \,.
\]
\end{itemize}
\end{corollary}

In particular we get:

\begin{theorem} Let $(G,H, \cL)$ be a CPS and $U$ a compact window. Then
\begin{itemize}
  \item[(a)]
\begin{align*}
\YY_{app}&=\YY_{b,\m}=\YY_{m}= \{ \Gamma \in \extU : \dens_{\cA}(\Gamma)= \dens(\cL) \theta_H(U) \} \,.
\end{align*}
\item[(b)]   $\m(Y_{app})=1$.
\item[(c)] For $(s,t)+\cL \in \TT$ we have $ (s,t)+\cL \in \mbox{Im}(\pi)$ if and only if $-s+\oplam(t+U)$ is a minimal density weak model set.
\item[(d)] $\theta_{\TT_{red}}(\mbox{Im}(\pi)) = 1$.
\item[(e)] The mapping $\pi : Y_{app} \to \TT_{red}$ is continuous.
\item[(f)] For $\Gamma \in \YY_{app}$ we have $\pi(\Gamma)= (s,t)+\cL + \{0 \} \times H_U$ if and only if there exists $(s',t') \in  (s,t)+\cL + \{0 \} \times H_U$
such that
\[
\Gamma \supseteq -s'+\oplam(t+W) \,.
\]
\item[(g)] $\pi$ induces a Borel factor from $(\extU, G, \m)$ to $(\TT_{red}, G)$.
\item[(h)] For each $(\chi, \chi^\star) \in \SSS$ the mapping $f_\chi( \Gamma):= (\chi, \chi^\star) \left(  \pi(\Gamma) \right)$ is an eigenfunction for $\chi$
  which is continuous on $\YY_{app}$.
\item[(i)] $\pi$ induces an isometric isomorphism
\[
\pi: L^2(\TT_{red}, \theta_{\TT_{red}}) \to L^2(\extU, \m)
\]
\end{itemize}
\end{theorem}

\section{A class of examples with Borel windows}

Consider a CPS $(G,H,\cL)$ and a precompact Borel window $B$.

Theorem~\ref{thm:main 2} gives the inclusion
\[
\YY_{\m} \subseteq \{ \Gamma \in \extB : \dens(\Gamma)= \dens(\cL) \theta_H(B) \}
\]

We have also seen that if $B$ is topological (meaning open or closed), the above inclusion is equality.

We give below a class of examples with Borel window for which this inclusion is strict.

we start with the following simple Lemma.

\begin{lemma}\label{l1} Let $(\RR^d, \RR^m, \cL)$ be any fully Euclidean CPS with one to one star mapping. Let $W \subset \RR^m$ be any regular model set and $\Gamma \subset \oplam(W)$ be any subset. Then, for each $0 <a <\infty$ there exists some precompact Borel set $B_a$ such that
\begin{align*}
  \Gamma &= \oplam(B_a) \\
  \lambda(B_a) &= a \,.
\end{align*}
Here $\lambda$ denotes the Lebesgue measure on $\RR^m$.
\end{lemma}
\begin{proof}

Define
$$W_0: = \Gamma^\star = \{ x^\star : \in \Gamma \} \subseteq W$$

Since the lattice $\cL$ is countable and hence so is $\pi_{\RR^m}(\cL)$. In particular, $W_0$ is at most countable and hence it is Borel and satisfies $\lambda W_0=0$. Moreover, since $\star$ is one to one we have
\[
\Gamma = \oplam(W) \,.
\]

Next, pick some $t$ such that $t+ [-\frac{\sqrt[m]{a}}{2}, \frac{\sqrt[m]{a}}{2}]^m \cap W = \emptyset$. Define
\[
W_1:= \left(t+ [-\frac{\sqrt[m]{a}}{2}, \frac{\sqrt[m]{a}}{2}]^m \right) \backslash \pi_{\RR^m}(\cL) \,.
\]
Then $W_1$ is Borel and
$$
\lambda(W_1) = a
$$
and
\[
\oplam(W_1)= \emptyset
\]

Then, the set $B_a=W_0 \cup W_1$ satisfies the desired conditions.

\end{proof}

Now, we can easily construct counterexamples. Pick for example the Fibonacci CPS $(\RR, \RR, \cL)$ and window, and let $\oplam(W)$ be the Fibonacci model set.
Pick a subset $\Gamma \subset \oplam(W)$ which is not pure point diffractive with respect to $A_n =[-n, n]$, for example a generic point set for the Bernoulisation of Fibonacci.

Set $a =\dens_{\cA}(\Gamma)$ and pick the window $B_a$ from Lemma~\ref{l1}. Then, by the construction of $B_a$, for this Borel window we have
\[
\Gamma \in \{ \Lambda \in \extB : \dens(\Lambda)= \dens(\cL) \theta_H(B) \} \,.
\]
On another hand, since every element in $\YY_{\m}$ has pure point diffraction, we have $\Gamma \notin \YY_{\m}$.

\smallskip

While the above example looks a bit artificial, heuristically any Borel set is obtained from a regular window via the addition/subtraction of closed sets, so in some sense, is constructed in a similar manner to the construction of $B_a$ in Lemma~\ref{l1}. Nevertheless, this example shows that in general we cannot have equality of these two sets.

\begin{remark} It is worth pointing out that in this example most of the claims in Theorem~\ref{Main 3} still hold.

Indeed, let $B_a$ be the set defined in Lemma~\ref{l1} and let $K= t+ [-\frac{\sqrt[m]{a}}{2}, \frac{\sqrt[m]{a}}{2}]^m$.
Then $\oplam(K)$ is a regular model set. Let $m$ be the unique ergodic measure on $\XX(\oplam(K))$.

It is easy to see that for all $y \notin \pi_{\RR^m}(\cL)$ we have
\[
\oplam(y+B_a)= \oplam(y+K) \,.
\]

Then, the set $T:= \{ (x,y) \in \TT : y \in \pi_{\RR^m}(\cL)\}$ has measure $0$ in $\TT$, and for all $(x,y)+\cL$ not in $\TT$  the set $-x+\oplam(y+B_a)= -x+\oplam(y+K)$ is generic for $m$.

From here, it follows immediately that
\[
\YY_{\m}=\YY_{b,m}= \YY_{app} = \XX(\oplam(K))
\]
and
\[
\m=m \,.
\]
Finally, $\pi : \YY_{app} \to \TT_{red}$ is just the torus parametrisation for the regular model set $\oplam(K)$.

From here all claims in Theorem~\ref{Main 3} excepting (f) and the last equality in (a) follow.

\end{remark}

\section{Open Questions}

We complete the paper by listing the following natural questions, to which we do not know the answer.

\begin{question} If $B$ is Borel, is it true that $m(\YY_{app})=1$?
\end{question}

\begin{question} Consider a Borel window. Is $\pi$ continuous? Is it a Borel map?
\end{question}

\begin{remark}
\begin{itemize}
  \item[(a)] By Theorem~\ref{thm:main} the range of $\pi$ has full measure in $\TT$ .
  \item[(b)] If $m(\YY_{app})=1$ and $\pi$ is a Borel map, it follows that $\pi$ induces a Borel factor $\pi: \YY_{app} \to \TT_{red}$, which induces an isometric isomorphism
  \[
  \pi : L^2(\TT_{red}) \to L^2(\extB, m) \,.
  \]
\end{itemize}
\end{remark}

\subsection*{Acknowledgments}  The work was supported by NSERC with grant  2020-00038. We are greatly thankful for all the support.

\end{document}